\documentclass[reqno, 10pt, a4paper]{amsart}

\usepackage{fullpage}
\selectfont

\usepackage{algorithm,algpseudocode}

\algnewcommand{\IIf}[1]{\State\algorithmicif\ #1\ \algorithmicthen}
\algnewcommand{\EndIIf}{\unskip\ \algorithmicend\ \algorithmicif}

\usepackage[utf8]{inputenc}
\usepackage[OT2,T1]{fontenc}
\usepackage[english]{babel}

\usepackage{amsmath}
\usepackage{amsfonts}
\usepackage{amssymb}
\usepackage{amsthm}

\usepackage{mathrsfs}
\usepackage{listings}
\usepackage[mathcal]{euscript}
\usepackage{bbm}%

\usepackage{enumitem}
\usepackage{multirow}

\usepackage[dvipsnames]{xcolor}
\usepackage{graphicx}

\usepackage{tikz-cd}
\usetikzlibrary{matrix,positioning}
\usepackage[all]{xy}

\usepackage{colortbl}
\definecolor{mygray}{gray}{0.92}

\usepackage{array}
\newcolumntype{C}[1]{>{\centering\arraybackslash$}p{#1}<{$}}

\usepackage{breqn}
\newcounter{myequation}[equation]

\usepackage{float}
\restylefloat{table}
\usepackage{multirow}

\usepackage{hyperref}

\theoremstyle{plain}
\newtheorem{theorem}{Theorem}[section]

\newtheorem{proposition}[theorem]{Proposition}
\newtheorem{lemma}[theorem]{Lemma}
\newtheorem{corollary}[theorem]{Corollary}

\theoremstyle{definition}
\newtheorem{definition}[theorem]{Definition}

\theoremstyle{remark}
\newtheorem{remark}[theorem]{Remark}
\newtheorem{example}[theorem]{Example}

\numberwithin{equation}{section}

\def\eps{\varepsilon}
\def\epsilon{\varepsilon}

\def\LG{L}
\def\LK{L^{Kani}}

\DeclareMathOperator{\ad}{adj}

\DeclareMathOperator{\disc}{disc}

\DeclareMathOperator{\End}{End}
\DeclareMathOperator{\Ends}{End^{sym}}

\DeclareMathOperator{\Gal}{Gal}

\DeclareMathOperator{\Gram}{Gram}

\DeclareMathOperator{\Hom}{Hom}
\DeclareMathOperator{\im}{im}
\DeclareMathOperator{\id}{Id}

\DeclareMathOperator{\Jac}{Jac}

\DeclareMathOperator{\NS}{NS}

\DeclareMathOperator{\Pic}{Pic}

\DeclareMathOperator{\res}{res}

\DeclareMathOperator{\Sp}{Sp}

\DeclareMathOperator{\tr}{tr}
\DeclareMathOperator{\trd}{trd}
\DeclareMathOperator{\Tr}{Tr}

\DeclareMathOperator{\rk}{rk}

\newcommand{\dq}{{/\kern -3pt/}}

\newcommand{\wm}{\boldsymbol{\omega}}

\def\C{\mathbb{C}}

\def\F{\mathbb{F}}

\def\N{\mathbb{N}}

\def\P{\mathbb{P}\,}
\def\Q{\mathbb{Q}}
\def\R{\mathbb{R}}

\def\Z{\mathbb{Z}}

\def\frakp{\mathfrak{p}}

\usepackage{bm}

\def\Ac{\mathcal{A}}
\def\Bc{\mathcal{B}}
\def\Cc{\mathcal{C}}
\def\Dc{\mathcal{D}}
\def\Ec{\mathcal{E}}
\def\Gc{\mathcal{G}}

\def\Lc{\mathcal{L}}

\def\Oc{\mathcal{O}}
\def\Pc{\mathcal{P}}
\def\Rc{\mathcal{R}}
\def\Vc{\mathcal{V}}

\def\Jv{\underline{J}}

\hypersetup{
  pdfauthor   = {Lercier, Liu, Lorenzo Garc\'ia, Ritzenthaler},
  pdftitle    = {Reduction type of smooth quartics},
  pdfsubject  = {},
  pdfkeywords = {},
  backref=true, pagebackref=true, hyperindex=true, colorlinks=true,
  breaklinks=true, urlcolor=blue, linkcolor=blue, citecolor=blue,
  bookmarks=true, bookmarksopen=true
}

\begin{document}

\title{An arithmetic intersection for squares of elliptic curves with complex multiplication}
\date{\today}

\begin{abstract}
Let $C$ be a genus $2$ curve with Jacobian isomorphic to the square of an elliptic curve with complex multiplication by a maximal order in an imaginary quadratic field of discriminant $-d<0$. We show that if the stable model of $C$ has bad reduction over a prime $p$ then $p \leq d/4$. We give an algorithm to compute the set of such $p$ using the so-called refined Humbert invariant introduced by Kani. Using results from Kudla-Rapoport and the formula of Gross-Keating, we compute for each of these primes $p$ its exponent in the discriminant of the stable model of $C$. We conclude with some explicit computations for $d<100$ and compare our results with an unpublished formula by the third author.
    \end{abstract}


\author[Lorenzo]{Elisa Lorenzo Garc\'ia}
 \address{
    Elisa Lorenzo Garc\'ia,
     Universit\'e de Neuch\^atel, IMA, Emile-Argand 11, 2000, Neuch\^atel,
    Switzerland. %
  }
\address{%
	Elisa Lorenzo Garc\'ia,
   Univ Rennes, CNRS, IRMAR - UMR 6625, F-35000
 Rennes, %
  France. %
}
\email{elisa.lorenzo@unine.ch, elisa.lorenzogarcia@univ-rennes1.fr}

\author[Ritzenthaler]{Christophe Ritzenthaler}
\address{%
	Christophe Ritzenthaler,
  Univ Rennes, CNRS, IRMAR - UMR 6625, F-35000
 Rennes, %
  France. %
  }
  
\address{%
	Christophe Ritzenthaler,
  Université Côte d'Azur, CNRS, LJAD UMR 7351,
  Nice,
  France
}
\email{christophe.ritzenthaler@univ-rennes1.fr}

\author[Rodr\'iguez Villegas]{Fernando Rodr\'iguez Villegas}
\address{%
 Fernando Rodr\'iguez Villegas,
 ICTP,
 Trieste,
 Italy}

\email{villegas@ictp.it}





\maketitle

\section{Introduction} \label{sec:intro}
The initial motivation of this article was the following formula appearing in \cite{villegas}. The third author defined there a genus $2$ curve 
\begin{align} \label{eq:villegaseq}
C & : y^2 =6^{-3} h_1(x) h_2(x) 
\end{align}
where 
\begin{align*} 
    h_1(x)& = (7144 \sqrt{-163}-151790) x^3+(129789 \sqrt{-163}+1752597)x^2+ \\ &  (-47481\sqrt{-163}+510153)x+(-1596\sqrt{-163}-37250)
\end{align*}
and $h_2(x)=x^3\overline{h(-1/x)}$ where bar denotes the complex conjugation of the coefficients. Its discriminant factors nicely\footnote{As we will see, in order to avoid \textit{ad hoc} factors like $2^{-12}$, one needs to consider not the discriminant of $C$ but an absolute normalized discriminant. See Example~\ref{ex:facdis} for the computation in these terms.} 
as
\begin{equation} \label{eqD} D=2^{-12} \disc(h_1h_2)= (2 \cdot 3^2 \cdot 5 \cdot 7 \cdot 11 \cdot 17 \cdot 19 \cdot 23)^{12}.
\end{equation}
The third author then observed that
\begin{equation} \label{eq:villegas163}
\log D=-6 \sum_{m \in \Z^3} \sum_{n | (163-Q(m))/4} \left(\frac{-163}{n}\right) \log n,    
\end{equation}
where $Q$ is the positive definite quadratic form with matrix $$\begin{pmatrix}
    24 & 4 & 6 \\ 4 & 55 & 1 \\ 6 & 1 & 83 
\end{pmatrix}.$$
The curve $C$ is of course special: its Jacobian is isomorphic to $E^2$, where $E$ is an elliptic curve with complex multiplication by $\Z[(1+\sqrt{-163})/2]$. The factorization of $D$ into a product of small primes is similar to the somehow `complementary' case of simple abelian surfaces with CM studied in full details in \cite{gola,gola2,goreng2,lauterviray0,lauterviray,lauterviray2}. But the formula~\eqref{eq:villegas163} is closer in essence to  the work of \cite{grosszagier} or in a different setting \cite{brunieryang}. 

In the 90', the third author actually developed a strategy that leads to such a formula for $E^2$ where $E$ has CM by a maximal order (see Section~\ref{sec:disc} for the explicit formula and comparison with the present results) but the work remained unfinished. However, the references and ideas it contains allowed us to shape the results in the present article and to propose an algorithm for the factorization of discriminants of genus 2 curves which Jacobian is the power of an elliptic curve with CM by a maximal order.

Let us now rewind and give a more formal introduction. Let $E$ be an elliptic curve with CM by the maximal order $\Oc$ of an imaginary quadratic field $K$ of discriminant $-d<0$. We consider a principally polarised abelian surface $A=E^2$ with an indecomposable polarization $\lambda$, i.e. there does not exists $E_1,E_2$ two elliptic curves such that $(A,\lambda)\simeq (E_1 \times E_2,\lambda_0)$ where $\lambda_0$ is the product polarization. Therefore $(A,\lambda)$ is the Jacobian of a genus 2 curve $C$ and one can prove (see beginning of Section~\ref{sec:algoprime}) that all these objects are defined over the Hilbert class field $H$ of $K$. Even better, given a prime $p \in \N$ and a prime $\frakp$ over $p$ in the ring of integers of $H$, we can construct the stable model of all these objects over its local completion $\Z_{\frakp}$ in 
$H_{\frakp}$. The model $\Ec$ of $E$ is smooth since $E$ has CM and therefore the special fiber $(\bar{E}^2,\bar{\lambda})$ of $(\Ec^2,\lambda)$ is a principally polarised abelian surface. However, when  $\bar{E}$ is supersingular, the polarization $\bar{\lambda}$ may become decomposable or equivalently the special fiber of the stable model of $C$ is singular. This shows the deep connection between decomposability and primes of bad reduction. We will first focus on writing down an algorithm to compute the primes $\frakp$ for which the polarization becomes decomposable. We call such a prime $\frakp$ a \emph{prime of decomposable reduction}.

The classical source  for studying decomposition of abelian surfaces (up to isogeny) over $\C$ is \cite{humbert}. It was revisited and extended to any characteristic in a series of papers by Kani and our main reference will be \cite{kanihumbert}. For any principally abelian surface $(A,\lambda)$,  he defines the so-called \emph{refined Humbert invariant} as an integral positive definite quadratic form $q_{(A,\lambda)}$ on the quotient $\LK:=\Ends(A,\lambda)/\Z \cdot 1_{\End(A)}$ of the symmetric endomorphism of $A$ (i.e. the endomorphisms which are invariant under the Rosati involution associated to $\lambda$). The crucial result for us is Theorem~\ref{th:kaniq}, \cite[Prop.6]{kanihumbert}: $(A,\lambda)$ is decomposable if and only if $q:=q_{(A,\lambda)}$ represents $1$. We give in Section~\ref{sec:kani} a summary of the results we need about this quadratic form which may be a good introduction for people interested in the prolific work of Kani. However, instead of working with $\LK$, we introduce a sublattice $\LG$ of $\Ends(A,\lambda)$, analogue to the Gross lattice for elliptic curve in \cite{grosslattice} such that $(\LK,q)$ and $(\LG,q/4)$ are isometric, but this second space allows a more direct proof of Theorem~\ref{th:kaniq}.

With this form at hand, a first piece of the algorithm becomes straightforward. In Section~\ref{sec:algoprime}, given 
\begin{itemize}
    \item an explicit $E/H$ with CM by $\Oc$,
    \item a indecomposable principal polarization given as $\lambda=\lambda_0 P$ where $P \in M_2(\Oc)$,
    \item a prime $p \in \N$ which is inert or ramified in $\Oc$,
    \item a prime ideal $\frakp \in H$ over $p$,
\end{itemize}  we  compute explicitly $\bar{\LG} \subset \Ends(\bar{E}^2,\lambda_0 P)$ and its Gram matrix with respect to $q/4$. It is then enough to look for the existence of an element of norm $1$ in the 5-rank lattice $\bar{\LG}$  to check whether $\frakp$ is a prime of decomposable reduction. This is summarised in Algorithm~\ref{algo:primereduction}.

The second step is to find a bound on the primes $p \in \N$ over which a decomposable reduction can occur. We use a key argument of \cite{kudla} by looking at the image of our lattice $\bar{\LG}$ into the kernel of the $p$-power Frobenius map on $\bar{E}^2$. Because the dimension drops, on can  show that $p$ has to divide the determinant of the Gram matrix $T$ of a rank-4 lattice generated by the rank-3 Gross lattice $\LG \subset \Ends(E^2,\lambda)$ and a vector of norm 1, necessarily in $\bar{\LG} \setminus \LG$. As one knows the discriminant of the Gram matrix of $\LG$ to be $4d$, one can then bound $p$ by $4d$. Adding some necessary congruence conditions on the remaining entries of the potential matrices $T$, one eventually gets the sharp bound $p \leq d/4$ (Theorem~\ref{th:finitep}).

Combining this two steps, one could then algorithmically obtain all the primes $\frakp$ of decomposable reduction for a given $(E^2,\lambda)$. In Section~\ref{sec:PDR}, we prefer to get a formulation free of the choice of $E$ among all Galois conjugate curves with CM by $\Oc$.  Inside the moduli stack $\Ac_2$ of principally polarised abelian schemes of relative dimension 2, we therefore  look at the $\Z_p$-cycle 
$$\mathcal{CM}(\Oc,P,p) := \bigsqcup (\Ec_{\sigma}^2,\lambda_0 P)$$
    where, for $\sigma \in \Gal(H/K)$, $\Ec_{\sigma}$ is the smooth model of $E^{\sigma}$  over the integer ring $\Z_{\frakp}$ of $H_{\frakp}$.    
 If we denote by $\mathcal{G}_1 \subset \Ac_2$ the locus of decomposable  abelian schemes, one can then define a proper arithmetic intersection between $\Gc_1$ and $\mathcal{CM}(\Oc,P,p)$  and therefore a multiplicity of intersection which we denote $e_{\Oc,P,p}$. This definition does not depend on the choice of $\frakp | p$, and we say that $p$ is a prime of \emph{potential decomposable reduction (PDR)} when $e_{\Oc,P,p}>0$. In Section~\ref{sec:support}, we first  work out the support of this intersection purely in terms of $P$ and the embeddings of $\End(E^{\sigma})$ into $\End(\overline{E^{\sigma}})$ for each $\sigma \in \Gal(H/K)$. Using results from \cite{colo} recently developed for post-quantum isogeny-based cryptography, one can actually run among these embeddings by weighing the embeddings of $\Oc$ into the conjugacy classes of maximal orders $\Rc$  in the definite quaternion algebra $\Bc_{p,\infty}$. The weight depends on $p$ being inert or ramified (with the special case $p=d/4$) and the two-side ideal $\Pc \subset \Rc$ over $p$ being principal or not (Lemma~\ref{lem:emb}). For each of these embeddings, one can then construct the corresponding $\bar{L}$  and test for the existence of $\ell_0 \in \bar{\LG} \setminus \LG$ such that $q/4(\ell_0)=1$. Moreover, the local analysis of \cite{kudla} applies and allows us to compute the intersection multiplicity $e_{\Oc,P,p}$ by adding the local contributions $e_{\ell_0}$ given by the Gross-Keating invariants of the Gram matrix associated to the 3-rank lattice $\ell_0^{\perp} \subset L \oplus \ell_0$ (Theorem~\ref{th:mult}). Algorithm~\ref{algo:PDR} summarizes these various results and is implemented in \texttt{Magma}.

Section~\ref{sec:disc} closes the loop and connect the computation of the primes $p$ of PDR and their intersection multiplicity $e_{\Oc,P,p}$  to the factorization of the discriminant of the curve $C$ whose Jacobian is one of the  Galois conjugates of $(E^2,\lambda_0 P)$. More precisely, for each prime $p$ of PDR and $\frakp \mid \ p$ in $H$, one can consider the stable model $\Cc$ of $C$ over $\Z_{\frakp}$. Igusa shows (see also Liu \cite[Th.1]{liug2}) that the special fiber of $\Cc$ being smooth or singular  can be characterised in terms of the valuations of the Igusa invariants $\Jv=(J_2,J_4,J_6,J_8,J_{10})$ of $C$ (recall that $J_{10}$ is the discriminant of $C$). More precisely, if we defined the \emph{normalised valuation of $J_{10}$ at $\frakp$} by $$v_{\Jv}(J_{10})= v_{\frakp}(J_{10})-10 \cdot \min\left(\frac{v_{\frakp}(J_{n})}{n}, \, n \in \{2,4,6,8,10\}\right),$$
we get that $\frakp$ is a prime of decomposable reduction if and only if $v_{\Jv}(J_{10}) > 0$. Proposition~\ref{prop:siegel} uses comparison formulas between Siegel modular forms and their expressions in terms of Igusa invariants to relate the local contribution $e_{\ell_0}$ at the level of the Jacobian to the normalized valuation of $J_{10}$ of the corresponding curve, exhibiting a factor 12 between the two. Finally, since moving in the Galois orbit of $E$ is equivalent to moving once (resp. twice) over the primes $\frakp$ over $p$ when $p$ is inert in $\Oc$ or $p=d/4$ (resp. ramified and different from $d/4$), we get that the \emph{exponent of the absolute normalised discriminant of $C$ at $p$}, $\sum_{\frakp | p} v_{\Jv}(J_{10})$, is equal to  $12  \cdot e_{\Oc,P,p}$ (resp. $\sum_{\frakp | p} v_{\Jv}(I_{10}) = 6  \cdot e_{\Oc,P,p}$), see Theorem~\ref{th:disc}. 

Section~\ref{sec:datum} offers experimental data for all $d<100$. Using the numerical algorithms \cite{sijsling-endo}, we compute uncertified equations of the curves $C$. For each of them, we could verify the equality of the exponent of their absolute normalised discriminant at $p$ with 6 (or 12) times $e_{\Oc,P,p}$. We also intend to compare the experimental results with the third author's unpublished formula. We could check that it shows very accurate predictions when $d$ is a prime. This is therefore a great motivation to come back to this incredible close formula in a future work. Then, it will hopefully be possible to accomplish the third author's initial dream and Kudla's program and relate it to the Fourier coeﬃcients of certain Siegel-Eisenstein series. The formula \cite[Th.8.4]{cho} seems to be a step in this direction.

How does our work compare to \cite{kudla}?  It seems to us that they consider a slightly different arithmetic intersection problem. Instead of looking at the cycle of $(E^2,\lambda)$ for a given $\Oc$ and $\lambda$, they consider the arithmetic intersection of $\mathcal{G}_1$ (for instance) with the set of points in the moduli space with a given $3 \times 3$ Gram matrix $Q$ for the lattice of special endomorphisms of one $(E^2,\lambda)$ with respect to $q/4$. The special endomorphisms are the elements of trace $0$ inside $\Ends(E^2,\lambda)$ and this lattice contains $\LG$. Our cycle $\mathcal{CM}(\Oc,P,p)$ is therefore a sub-cycle of the one they consider. Looking at formula (3.6) of \textit{loc. cit.}, one may get the impression that this is much more natural as all extensions of $Q$ to $4\times 4$ positive definite matrices $T$ with the last entry equal to 1 give an intersection point. However, this first impression is not  correct as for instance for the second polarization in the case $d=163$, there is a matrix $T$ of determinant $18$ with Gross-Keating invariants $[0,0,2]$ at $p=3$ and hence the formula  \eqref{eq:GK} does not apply. Our algorithms would allow to compute the exact support of the intersection in their scenario as well.

In our case as well, the set of matrices $T$ has still some mysteries and one final observation in Section~\ref{sec:observations} confirms that we just scratch the surface of this problem. Indeed, in all our experiments, the inert primes $p$ of PDR seem to be the inert primes dividing the determinants of the positive definite matrices $T$  constructed from the Gram matrice of $L$ and $1$ as above and satisfying the congruence conditions \eqref{eq:sieve}. We do not know how to prove this.

Finally, there are many possible  extensions of that work. First, note that it may be feasible using the same techniques to handle the case of non-maximal orders, the case where $\Jac(C) \sim E^2$ or to look at the similar problem for $g=3$ (see Remark~\ref{rem:g3} and Remark~\ref{rem:boundg3}).

\subsection*{Acknowledgements} The research of the first author is partially funded by the
Melodia ANR-20-CE40-0013 project. The research of both authors is partially funded by the 2023-08 Germaine de Staël project. 

We would like to thank Jeroen Sijsling and Edgar Costa for their kind assistance with their \texttt{Magma} packages. We also thank Everett Howe for the quick fix on the enumeration of polarizations. We are very grateful to David Kohel for his patient explanation on orientations on supersingular elliptic curves.

\subsection*{Convention and notation} All objects and morphisms are relative to a base ring. For instance, if $A$ is defined over a perfect field $k$, $\End(A)$ refers to the endomorphisms of $A$ over $k$. The adjective `geometric' will be added to refer to the ones over $\bar{k}$.
\begin{itemize}
    \item $\Oc=\Z[\omega]$ is the maximal order of an imaginary  quadratic field $K$ of discriminant $-d<0$.
    \item $\Rc$ is a maximal order of the definite quaternion algebra $\Bc_{p,\infty}$ of discriminant $p$. We denote by $\Pc$ the two-sided prime ideal of $\Rc$ above $p$.
    \item $H$ is the Hilbert class field of $K$ and for a prime $p \in \N$ and a prime ideal $\frakp | p$ in $H$, we denote $H_{\frakp}$ the completion at $\frakp$, $\Z_{\frakp}$ its ring of integers with valuation $v$ and $k_{\frakp}$ its residue field.
    \item For a principally polarized abelian variety $(A,\lambda)$ of dimension $g$, we denote $\Ends(A,\lambda)=\{\alpha \in \End(A),  \, \alpha=\alpha^{\dag}\}$ where $\dag$ is the Rosati involution associated to $\lambda$. We denote $\LK=\Ends(A,\lambda)/\Z \cdot 1_{\End(A)}$ and $\LG=\tau(\Ends(A,\lambda))$ where  $\tau(\alpha)=g\alpha-\Tr(\alpha)/2$.
    \item For an elliptic curve $E$ with CM by $\Oc$, $\lambda_0$ is the product principal polarization on $E^2$ and for a principal polarization $\lambda$, we let
    $\lambda_0^{-1} \lambda=P \in M_2(\Oc)$ the associated hermitian matrix.
    \end{itemize} 

\section{Quadratic forms on symmetric endomorphisms} \label{sec:kani}
All the material in this section is classical or covered by Kani in his various articles, in particular in \cite{kanihumbert, KaniII, KaniI}. However the presentation and some proofs are original.

\subsection{Refined Humbert invariant}
Let $(A,\lambda)$ be a principally polarised abelian variety of dimension $g$ over a perfect field $k$. The \emph{Rosati involution} (with respect to $\lambda$) is defined on $\End(A)$ by the involution $\alpha \mapsto \alpha^{\dag} = \lambda^{-1} \hat{\alpha} \lambda$.  We will denote $\Ends(A,\lambda)=\{\alpha \in \End(A),  \, \alpha=\alpha^{\dag}\}$. 
To any $\alpha \in \End(A)$, one can associate a degree $2g$ polynomial $P_{\alpha}\in \Z[x]$ defined for all $n \in \Z$ by $P_{\alpha}(n)=\deg(\alpha-n)$. By \cite[Prop.12.9, p.125]{cornell}, it is also the characteristic polynomial of $\alpha$ acting on $T_{\ell}(A) \otimes \Q$ for all $\ell$ different from the characteristic of $k$ and if $k \subset \C$, it is also the characteristic polynomial of the rational representation of $\alpha$. The \emph{trace} of $\alpha$ on this space does not depend on $\ell$, belongs to $\Z$ and will be denoted $\Tr(\alpha)$.

By \cite[Th.17.3, p.138]{cornell}, the bilinear form $(\alpha,\beta) \mapsto \Tr(\alpha \beta^{\dag})$ on $\End(A)$ is positive definite and restrict to a positive definite quadratic form $\Ends(A,\lambda) \to \Z$ by $\alpha \mapsto \Tr(\alpha^2)$. Similarly,
Kani introduced another quadratic form on $\Ends(A,\lambda)$:
\begin{definition}
 One calls     
 \begin{eqnarray*}
q_{(A,\lambda)} : & \Ends(A,\lambda) & \to \Z \\
      & \alpha & \mapsto \frac{1}{4} (2g \Tr(\alpha^2) - \Tr(\alpha)^2)
 \end{eqnarray*}
 the \emph{refined Humbert invariant}.
\end{definition}

\begin{remark}
The fact that $q_{(A,\lambda)}$ has values in $\Z$ boils down to the fact that $\Tr(\alpha)$ is even for symmetric endomorphisms    (see Remark~\ref{rem:integral}).
\end{remark}

Note that for any integer $m$, one has $q_{(A,\lambda)}(\alpha+m)=q_{(A,\lambda)}(\alpha)$ so that this form is well defined on the quotient $$\LK:=\Ends(A,\lambda)/\Z \cdot 1_{\End(A)}.$$ Moreover letting $\alpha'=2g \alpha-\Tr(\alpha) \cdot 1_{\End(A)}$, we see that $4g^2 \cdot q_{(A,\lambda)}(\alpha)= q_{(A,\lambda)}(\alpha') = g/2 \cdot \Tr({\alpha'}^2) \geq 0$, so the form $q_{(A,\lambda)}$ is a positive definite form on $\LK$.

\subsection{Gross lattice}
Instead of working on this quotient lattice, one can introduce another lattice 
$$\LG:=\tau(\Ends(A,\lambda)) \; \textrm{where }  \tau(\alpha)=g\alpha-\Tr(\alpha)/2,$$ which we call the \emph{Gross lattice} by analogy with the genus 1 case \cite{grosslattice}. Clearly $\Tr(\tau(\alpha))=0$ so $\LG \subset L_0:=\{\alpha \in \Ends(A,\lambda), \, \Tr(\alpha)=0\}$, the latter being called in the literature the space of \emph{special endomorphisms}. The kernel of $\tau$ is $\Z\cdot 1_{\End(A)}$ so for $\beta \in \LG$, the image in $\LK$ of $\alpha=\tau^{-1}(\beta)$ is well defined and realises an isomorphism between $\LG$ and $\LK$.  We therefore have the following relation 
\begin{center}
\begin{tikzcd}
\End(A)  \arrow[hookleftarrow]{r}
  & \Ends(A,\lambda) \arrow[hookleftarrow]{r}  \arrow[twoheadrightarrow]{d} \arrow[rr, bend left=20, "\tau"]
  & L_0 \arrow[hookleftarrow]{r}
  & \LG \arrow[lld, dotted, "iso", description]
    \\
  & \LK   &    &  
\end{tikzcd}
\end{center}
Moreover $$q_{(A,\lambda)}(\beta)= \frac{g}{2} \Tr(\beta^2) = g^2 q_{(A,\lambda)}(\alpha).$$
Hence, the lattice $(\LG,\frac{1}{g^2} \cdot q_{(A,\lambda)})$ and $(\LK,q_{(A,\lambda)})$ are isometric and in particular their Gram matrices are equivalent.

\subsection{The case $g=2$}
From now on we restrict to $g=2$ so that $q_{(A,\lambda)}(\alpha) = \Tr(\alpha^2)-\frac{1}{4} \Tr(\alpha)^2$. We let $q:=q_{(A,\lambda)}$ when there is no ambiguity and let $$b_q: (\alpha_1,\alpha_2) \mapsto \frac{1}{2} \left( q(\alpha_1+\alpha_2)-q(\alpha_1)-q(\alpha_2)\right)$$ be the bilinear form associated to $q$.
We want to prove several integrality and congruence results on $q$ and $b_q$. In order to do this, we will use Kani's expression of $q$ in terms of intersection of divisors.

One embeds the N\'eron-Severi group $\NS(A)$ in $\Ends(A,\lambda)$ by sending a divisor $D$ to $\alpha={\lambda}^{-1} \phi_D$, where $\phi_D(x) = t_x^* \Lc_D \otimes \Lc_D^{-1}$. Over $\bar{k}$, the map is surjective and in particular the identity of $\Ends(A,\lambda)$ corresponds to an effective divisor $\theta$ such that $\phi_{\theta}=\lambda$.  If $\alpha \in \Ends(A)$ corresponds to $D$, then the degree of $\alpha$ is equal to the degree of $\phi_D$  since ${\lambda}$ is an isomorphism. By the Geometric Riemann-Roch theorem, one has $\deg \alpha = (\frac{1}{2!} (D.D))^2$. In particular $(\theta.\theta)=2$.
One rewrites $q$  in terms of intersection of divisors on $A$ as in \cite{kanihumbert}:
\begin{proposition} \label{prop:comparison}
For $D \in \NS(A)$, let $\alpha= \lambda^{-1} \phi_D \in \Ends(A,\lambda)$. Then
\begin{equation} \label{eq:divisor}
q(\alpha) =(D.\theta)^2-2(D.D).    
\end{equation}
\end{proposition}
The expression \eqref{eq:divisor} is what Kani actually calls the \emph{refined Humbert invariant}.
This equality follows from the following lemmas.

\begin{lemma} \label{lem:repalpha}
Let $(A,\lambda)$ be a principally polarised abelian variety over $k$ of dimension $g$ and  $\alpha \in \Ends(A,\lambda)$. Let $\ell$ be a prime different from the characteristic of $k$. There exists a basis of $T_{\ell}(A)$ such that the matrix of $\alpha$ acting on $T_{\ell}(A)$ in this base can be written 
\begin{equation} \label{eq:matM}
    M_{\alpha}=\begin{pmatrix}
\alpha_1 \id_2 & M_{12} & \ldots & M_{1g} \\
\ad(M_{12}) & \ddots & \cdots & M_{2g} \\
\vdots & \ddots  & \ddots & \vdots \\
\ad(M_{1g}) & \cdots & & \alpha_g \id_2 
\end{pmatrix}
\end{equation}
where $\alpha_1,\dots,\alpha_g \in \Z_{\ell}$ and $M_{ij} \in M_2(\Z_{\ell})$ for $i<j$ and $\ad()$ is the adjucate operator.
\end{lemma}
\begin{proof}
Let $e_{\ell}^\theta : T_{\ell}(A) \times T_{\ell}(\hat{A}) \to \Z_{\ell}(1)$ be the Weil pairing relative to the polarization $\lambda$ associated to $\theta$, i.e. $e_{\ell}^{\theta}(x,y)= e_{\ell}(x,\lambda(y))$.  
Fix a basis $a_1,\ldots,a_{2g}$  of 
$T_{\ell}(A)$ such that $e_{\ell}^{\theta}$ is given by the diagonal block matrix $J,\cdots,J$ where 
$$
J=\begin{pmatrix}0 & 1 \\ -1 & 0 \end{pmatrix}.
$$
Write the matrix of $\alpha$ acting on $T_{\ell}(A)$ with $2\times2$-blocks as
$$M_{\alpha}=\begin{pmatrix}
M_{11} & M_{12} & \ldots & M_{1g} \\
M_{21} & \ddots & \cdots & M_{2g} \\
\vdots & \ddots  & \ddots & \vdots \\
M_{g1} & \cdots & \cdots & M_{gg} 
\end{pmatrix}.$$
Since $\alpha$ is a symmetric endomorphism for the polarization $\lambda$,  the following holds: $$e_\ell^\theta(a_i,\alpha(a_j))=e_\ell^\theta(\alpha^\dag(a_i),a_j)=e_\ell^\theta(\alpha(a_i),a_j)=-e_\ell^\theta(a_j,\alpha(a_i)).$$
This equality
implies $JM_{ij}=-M_{ji}J$. Since they are $2\times2$-matrices, this implies $M_{ji}=\ad(M_{ij})$, which, in addition, for $i=j$, implies that $M_{ii}$ are multiples of the identity.

\end{proof}

\begin{lemma} \label{lem:atoD}
Let $(A,\lambda)$ be a principally polarised abelian surface. For any $\alpha \in \Ends(A,\lambda)$  with associated divisor $D$ we have that $\Tr(\alpha) = 2 (\theta.D)$ and $\frac{\Tr(\alpha)^2}{2}- \Tr(\alpha^2)= 2 (D.D)$.
\end{lemma}

\begin{proof}
We resume the notation of Lemma~\ref{lem:repalpha} in the case  $g=2$ and write 
\begin{equation} \label{eq:reprag2}
M_{\alpha}=\begin{pmatrix}\alpha_{11} & 0& \alpha_{13} & \alpha_{14}\\ 0 & \alpha_{11} & \alpha_{23} & \alpha_{24}\\ \alpha_{24} & -\alpha_{14} & \alpha_{33} & 0\\ -\alpha_{23} & \alpha_{13} & 0 & \alpha_{33}\end{pmatrix}.
\end{equation}
We adapt the proof of  \cite[Th.17.3]{milne}. By Lemma~17.4 in \textit{loc. cit.}, there exists a canonical generator $\epsilon$ of $\Hom(\Lambda^4(T_{\ell}(A)),\Z_{\ell}(2))$ such that for any divisors $D_1,D_2$ one has $e_{\ell}^{D_1} \wedge e_{\ell}^{D_2} = (D_1.D_2) \cdot \epsilon$ where $e_{\ell}^{D_i}(x,y) = e_{\ell}(x, \phi_{D_i} y)$.
We apply this lemma to $D_1=\theta$ and $D_2=D$ and then to $D_1=D$ and $D_2=D$. 

We can compute the following expression using that the matrices of $e_{\ell}^{\theta}$ and $e_{\ell}^D$ in the chosen basis $(a_1,\ldots,a_4)$ are respectively 
$$\begin{pmatrix}
    J & 0 \\ 0 & J
\end{pmatrix} \, \textrm{and} \begin{pmatrix}
    0 & -\alpha_{11} & -\alpha_{23} & -\alpha_{24} \\
    \alpha_{11} & 0 & \alpha_{13} & \alpha_{14} \\
    \alpha_{23} & \alpha_{13} & 0 & -\alpha_{33} \\
    \alpha_{24} & \alpha_{14} & \alpha_{33} & 0 
\end{pmatrix}.$$
We get
\begin{eqnarray*}
&(e_\ell^\theta\wedge e_\ell^\theta)(a_1\wedge...\wedge a_4)=\frac{1}{24}\sum_{\sigma\in S_4} \textrm{sg}(\sigma) \cdot e_\ell^\theta(a_{\sigma(1)}, a_{\sigma(2)}) \cdot e_\ell^\theta(a_{\sigma(3)}, a_{\sigma(4)})=\\
&\frac{1}{24}\left(8e_\ell^\theta(a_1, a_2)e_\ell^\theta(a_3, a_4)-8e_\ell^\theta(a_1, a_3)e_\ell^\theta(a_2, a_4)+8e_\ell^\theta(a_1, a_4)e_\ell^\theta(a_2, a_3)\right)=\\
&\frac{1}{3}e_\ell^\theta(a_1, a_2)e_\ell^\theta(a_3, a_4)=-\frac{1}{3}=(\theta.\theta) \cdot \epsilon(a_1\wedge ...\wedge a_4).
\end{eqnarray*}
Since $(\theta.\theta)=2$, we get $\epsilon(a_1\wedge ...\wedge a_4)=-1/6$. We can now compute
\begin{eqnarray*}
-(\theta.D)/6=(e_\ell^\theta\wedge e_\ell^D)(a_1\wedge...\wedge a_4)=\\
\frac{1}{24}\left(4e_\ell^\theta(a_1,a_2)e_\ell^D(a_3, a_4)+4e_\ell^\theta(a_3, a_4)e_\ell^D(a_1, a_2)\right)=\\
\frac{1}{12}\left(2 e_{\ell}^D(a_1,a_2) -2 e_{\ell}^D(a_3,a_4)\right)=-\frac{1}{12}\Tr(\alpha).
\end{eqnarray*}
Analogously, a more tedious computation leads to  $$-(D.D)/6=(e_\ell^D\wedge e_\ell^D)(a_1\wedge...\wedge a_4)= \frac{1}{6} \left(\Tr(\alpha^2) - \frac{\Tr(\alpha)^2}{2}\right).$$
\end{proof} 

Using that $(D.D)=\pm 2\sqrt{\deg(\alpha)}$ is even and  \eqref{eq:divisor}, one can prove the following results.
\begin{corollary} \label{cor:bqintegral} 
  One has
 \begin{enumerate}
     \item for any $\alpha \in \Ends(A,\lambda)$, $q(\alpha) \in \Z$ and $q(\alpha)$ is congruent to $0$ or $1$ modulo 4.
     \item For all $\alpha_1,\alpha_2 \in \Ends(A,\lambda)$, with corresponding divisors $D_1$ and $D_2$, $$b_q(\alpha_1,\alpha_2) = (D_1 . \theta) \cdot (D_2 . \theta) -2 \cdot (D_1 . D_2)  \in \Z.$$      In particular the bilinear form $b_q$ (resp. $b_q/4$) is integral on $\LK$ (resp. on $\LG$).
     \item For $\alpha_1,\alpha_2$ in $\Ends(A,\lambda)$, $$b_q(\alpha_1,\alpha_2)^2 \equiv q(\alpha_1) \cdot q(\alpha_2) \pmod{4}.$$ \label{cor:cong}
 \end{enumerate}
\end{corollary}

\begin{remark} \label{rem:integral}
    The fact that $q_{(A,\lambda)} = \frac{1}{4} (2g \Tr(\alpha^2) - \Tr(\alpha)^2)$ is  in $\Z$ is actually true in for any $g \geq 1$ and boils down to the fact that $\Tr(\alpha)$ is even for symmetric endomorphisms. Using \eqref{eq:matM}, this can be  easily proved in any characteristic different from $2$. A global argument can be given using Proposition~(12.29) of the unpublished book of Edixhoven, van der Geer and Moonen on abelian varieties with \cite[Prop.12.12,p.126]{cornell}. We sketch the proof here for completeness.  We can assume that $A$ is simple. The multiplicative function $\chi : \NS(A) \otimes \Q \to \Q$ which associates to $D$ its Euler characteristic is an homogeneous polynomial of degree $g$. It induces then a homogenous polynomial function of degree $g$ on $F=\Ends(A,\lambda) \otimes \Q$ which is a power of the reduced norm on $K$. Since this norm is a polynomial of degree $[F:\Q]$, we see that $f=[F:\Q]$ divides $g$.  One can then apply \cite[Prop.12.12,p.126]{cornell} with $d=1$, to deduce that $\Tr(\alpha)=2g/f \cdot \trd(\alpha) \in 2\Z$, where $\trd$ is the reduced trace on $F$.
\end{remark}

\begin{remark} \label{rem:minpol}
Going back to $g=2$, let $\alpha \in \Ends(A,\lambda)$ and
$$M_{\alpha}=\begin{pmatrix}
    a \id_2 & M \\ \ad(M) & b \id_2
\end{pmatrix}$$
its matrix on $T_{\ell}(A)$. Then $\alpha$ satisfies a degree two minimal polynomial, namely $X^2-(a+b) X + (ab-\det(M)) \in \Z[X]$. Indeed, notice that  $a+b=\tr(M_{\alpha})/2=\Tr(\alpha)/2 \in \Z$ and that $(ab-\det(M))^2=\det(M_{\alpha})=\deg(\alpha)=\frac{1}{4} (D.D)^2 \in \Z^{\times 2}$.  
\end{remark}

\begin{remark}
There is an direct relation with the classical Humbert invariant $\Delta$ defined for  abelian surfaces with real multiplication over $\C$ (see for instance \cite{hashimoto-form} for the definition of the Humbert invariant). Using the formula on p.542 in \textit{loc. cit.} for the coefficients of the matrix of $\alpha \in \Ends(A,\lambda)$ in the rational representation, we see that $\Delta=q(\alpha)$ (see also 
\cite[Cor. 9.2.3]{kirthesis}).
\end{remark}

\subsection{Decomposable abelian surfaces}
Kani gives a generalization in all characteristics of the  characterization of decomposable abelian surfaces initially proved in terms of Humbert invariant over $\C$ for abelian surfaces with real multiplication.
\begin{theorem}[{\cite[Prop.6]{kanihumbert}}] \label{th:kaniq}
Let $(A,\lambda)$ be a principally polarised abelian surface over a field $k$. Then $(A,\lambda)$ is decomposable over $k$ if and only if there exists $\alpha \in \LK=\Ends(A,\lambda)/1_{\End(A)}$ (resp. $\beta \in \LG=\tau(\Ends(A,\lambda))$) such that $q(\alpha)=1$ (resp. $q/4(\beta)=1$).
\end{theorem}
Kani's proof relies on several results on the divisors of abelian surfaces. We would like to give here a different proof based on direct computations with endomorphisms.  An element $\beta=\tau(\alpha) \in \LG$ (and more generally in $L_0$) has a representation on $T_{\ell}(A)$ of the form 
$$\begin{pmatrix}
    b & M \\ M^{\ad} & -b
\end{pmatrix}.$$
In particular $\beta^2 = (b^2+\det(M)) \cdot 1_{\End(A)}$. Now, $q(\beta) = 4(b^2+\det(M)) = 4 q(\alpha)$, so we see that $\beta^2=q(\alpha) \cdot 1_{\End(A)}$. Hence, we can again reformulate the statement by asking alternatively for the existence of $\beta \in \LG$ such that $\beta^2=1_{\End(A)}$. If $(A,\lambda)$ is decomposable, it is isomorphic to $E_1 \times E_2$ with the product polarization $1 \times 1$ and $$\beta=\begin{pmatrix}
    1 & 0 \\ 0 & -1
\end{pmatrix} = \tau\left(\begin{pmatrix}
    1 & 0 \\ 0 & 0 
\end{pmatrix}\right) \in \LG \subset \Ends(E_1 \times E_2, 1 \times 1)$$ satisfies the statement. Conversely, assume that we can find such a $\beta=\tau(\alpha) \in \LG$. Since  $q(\alpha)=1=\Tr(\alpha^2)-(\Tr(\alpha)/2)^2$, we see that $\Tr(\alpha)/2$ must be odd since $\Tr(\alpha^2)$ is even. So $\beta-1=2\alpha-(\Tr(\alpha)/2-1) \cdot 1_{\End(A)}$ is divisible by $2$. This proves that $e_1=(\beta+1_{\End(A)})/2$ and $e_2=(1_{\End(A)}-\beta)/2$ are two endomorphisms of $A$ such that $e_1e_2=0$. Hence $E_1=\im(e_1) \subset \ker(e_2)$ is a proper subvariety of $A$, hence an elliptic curve and similarly for $E_2=\im(e_2)$. Moreover the morphism $x \mapsto (e_1(x),e_2(x))$ is an isomorphism from $A$ to $E_1 \times E_2 \subset A$ with inverse $e_1+e_2$. It remains to show that the polarization $\lambda$ is decomposable: 
computing the Weil pairing for $(x_1,0) \in E_1 \times E_2$ and $(0,x_2)\in E_1 \times E_2$, we see that $e^{\lambda}((x_1,0),(0,x_2))=e^{\lambda}(e_1 (x_1,0),e_2 (0,x_2)) = e^{\lambda}((x_1,0),e_1 e_2 (0,x_2))=1$ since $e_1e_2=0$.

\begin{remark} \label{rem:g3}
What about $g>2$? If $(A,\lambda)\simeq(A_1\times A_2,\lambda_1\times\lambda_2)$ where $(A_i,\lambda_i)$ are principally polarised abelian varieties of dimension $g_i$, then the symmetric endomorphism $\alpha=\operatorname{diag}(1,\dots,1,0,\dots,0)$ satisfies $q_{(A,\lambda)}(\alpha) = \frac{1}{4} (2g \Tr(\alpha^2) - \Tr(\alpha)^2)=g_1g_2$. This is therefore a necessary condition. 

Conversely, a symmetric endomorphism $\alpha$ with $q(\alpha)=g_1g_2$ defines an endomorphism $\beta=g\alpha-\Tr(\alpha)/2 \cdot 1_{\End(A)}$ with $$\Tr(\beta)=0, \; \Tr(\beta^2)=2gg_1g_2 \; \textrm{and } q(\beta)=g^2g_1g_2.$$ The endomorphism $e_1=\frac{\beta+g_1 \cdot 1_{\End(A)}}{g}$ is well-defined because $(\Tr(\alpha)/2)^2\equiv {g_1}^2\mod{g}$ and we can take $-\alpha$ instead of $\alpha$ if needed. We also define $e_2=1_{\End(A)}-e_1$. It remains to check whether $e_1e_2=0$: this happens if and only if $\beta^2+(g_1-g_2)\beta-g_1g_2=0$.
This is always the case for $g=2$ and $g_1=1$. For $g=3$ and $g_1=1$ for instance,  Lemma \ref{lem:repalpha} only implies that we have $\beta^3-3\beta-\Tr(\beta^3)/6=0$. So, the condition $\beta^2-\beta-2=0$ has to be checked. If $\Ends(A,\lambda)$ can be explicitly computed  (as we will see in Section \ref{sec:algoprime} this is the case for $A=E^2$ and more generally for $E^g$), this computation is doable since there will be only a finite number of $\beta \in \LG$ with $(q/g^2)(\beta)=g_1g_2$ since $q/g^2$ is a positive definite form on the lattice $\LG$. In addition, since $\beta$ is self-adjoint, it diagonalises and it is equivalent to check whether $\Tr(\beta^3)=12$. 
For $g>3$, in order to check decomposability, it is not enough to check for $g_1=1$ (i.e. for an elliptic factor) and one needs to consider all possibilities for $g_1=1,\dots,\lfloor \frac{g}{2}\rfloor$. 

\end{remark}

\section{An algorithm to determine if $(E^2,\lambda)$ has decomposable reduction at a given prime} \label{sec:algoprime}

We restrict now to the case $A=E^2$  where $E$ is an elliptic curve over a number field having complex multiplication (CM) by the maximal order $\Oc=\Z[\omega]$ of discriminant $-d<0$ of an imaginary quadratic field $K$  with non-trivial involution $\bar{\phantom{a}}$. The elliptic curve $E$ can be defined over the Hilbert class field $H$ of $K$. We denote by $\lambda_0$ the product polarization on $E^2$ and by $\lambda$ any principal polarization defined over $\C$. The abelian surface $E^2$ is defined over $H$ and also all its geometric endomorphisms. Hence the polarization $\lambda=\lambda_0 (\lambda_0^{-1} \lambda)$ is actually defined over $H$. Since $\Ends(E^2,\lambda)=\Ends(E_{\bar{H}}^2,\lambda)$, we get that $\NS(A_{\bar{H}})=\NS(A)$ and by Theorem~\ref{th:kaniq}, $\lambda$ is geometrically indecomposable if and only if it is indecomposable. Hence, if $\lambda$ is indecomposable, $(E^2,\lambda)$ is the Jacobian of a genus $2$ curve $C$, also defined over $H$. Indeed, the possible obstruction for $C$ to be defined on $H$ is measured by the map $\mu$ in the exact sequence $$\hat{A}(H) \to \Pic(A)(H) \to \NS(A_{\bar{H}})^{\Gal(\bar{H}/H)} \xrightarrow{\mu} H^1(\Gal(\bar{H}/H),A(\bar{H}))$$ \cite[Rem.13.2, p.127]{cornell}. Now, as explained in \cite[Appendix~B.1]{silverman}, for $[D] \in \NS(A_{\bar{H}})$, $\mu([D])$ is defined by mapping a preimage $D \in \Pic(A_{\bar{H}})$ of $[D]$  to the class of $\zeta_{\sigma} : \sigma \mapsto  D^{\sigma}-D$. But there exist explicit representatives in $\Pic(A_{\bar{H}})$ of generators of $\NS(A_{\bar{H}})=\NS(A)$ and there are defined over $H$ (see \cite[Prop.23]{kanihumbert}). So $\zeta_{\sigma}=0$ for all $\sigma$ and the map $\mu$ is the zero map. Hence there is no obstruction.

 Let $p \in \mathbb{N}$ be a prime and  $\frakp$  be a prime ideal over $p$ in the ring of integers of $H$. We recall now the decomposition of $p$ in $H$.

 \begin{lemma}\label{lem:decomp} (See Thm. 3.3 in \cite{llo}) Let $K=\mathbb{Q}(\sqrt{-d})$ with $d>0$ a fundamental discrimimant, $\Oc$ be its maximal order and $H$ be its Hilbert class field. Let $h=\textrm{Cl}(O)=[H:K]$.  If $p$ is inert in $\Oc$ then $p$ decomposes as $\frak{p}_1\dots\frak{p}_h$ in $H$. If $p\mid d$ and $d=-p$ or $-4p$, then $p=\frak{p}_1^2\dots\frak{p}_h^2$. If $p\mid d$ and $d\notin\{-p,-4p\}$, then $h$ is even and $p=\frak{p}_1^2\dots\frak{p}_{h/2}^2$
 \end{lemma}
 
 Let $H_{\frakp}$ be the completion of $H$ at $\frakp$. Since $E$ has CM, it has potentially good reduction everywhere and by \cite[Cor.5.22]{rubinCM}, there actually exists a model $\Ec$ of $E$ over the ring of integers $\Z_{\frakp}$ of $H_{\frakp}$ which has good reduction $\bar{E}$ over the residue class field $k_{\frakp}$. In the sequel, we always suppose that we work with this model. We also denote by $\bar{\lambda}$ the reduction modulo $\frakp$ of the extension of $\lambda$ to $\Ec^2$. 
 
\begin{definition}
We say that $(A,\lambda)=(E^2,\lambda)$ has \emph{decomposable reduction} at $\frakp$ (or that $\frakp$ is a prime of decomposable reduction for $(A,\lambda)$) if the special fiber $(\bar{E}^2,\bar{\lambda})$ of $(\Ec^2,\lambda)$ over $\Z_{\frakp}$ is decomposable, i.e. if there exists two elliptic curves $\bar{E}_1,\bar{E}_2$ over $k_{\frakp}$ such  $(\bar{A},\bar{\lambda}) \simeq (\bar{E}_1 \times \bar{E}_2,1 \times 1)$, where $1 \times 1$ is the product polarization on $\bar{E}_1 \times \bar{E}_2$.  
\end{definition}
Equivalently, $(E^2,\lambda)$ has decomposable reduction at $\frakp$ if and only if the reduction of the stable model of the curve $C$ over $\Z_{\frakp}$ is not smooth, more precisely, the union of the two elliptic curves $\bar{E}_i$ intersecting at one point.\\

One can check algorithmically if a prime is of decomposable reduction. We now explicitly describe the various elements we need. Our running example will correspond to the example given in the introduction with $d=163$, i.e. $\omega=\frac{1+\sqrt{-d}}{2}$, which is of class number $1$, hence up to $\bar{\Q}$-isomorphism, the curve $E$ is unique.

\subsection{Polarization} \label{sec:pol}  Let $\lambda_0$ be the product polarization on $E^2$. One can represent the principal polarization $\lambda$ by a hermitian positive definite matrix $P=\lambda_0^{-1} \lambda=\begin{pmatrix} a & b \\ \bar{b} & c \end{pmatrix} \in M_2(\Oc)$ of determinant $1$ with $a,c \in \Z$. Given an order $\Oc$, there exist algorithms to produce all principal polarizations $\lambda$ on $E^2$. This was done in \cite{gelin}\footnote{Note that if \cite{gelin} gives all the principal polarizations, it sometimes wrongly states that some are indecomposable (but never for the cases of interest in the rest of the article). A corrected version for the present article was kindly provided by E. Howe.} and more generally for $A \sim E^2$ and $\Oc$ not necessarily maximal in \cite{KNRR}. Note that the smallest discriminant $d$ for which an indecomposable principal polarization exists is $d=8$. In the sequel, we will refer to $P$ by $[a,c,b]$.

\begin{example}
There are 7 indecomposable principal polarizations on $E^2$ where $E$ has CM by an order of discriminant $-d=-163$ and the one corresponding to  example~\eqref{eq:villegaseq} is $P=\begin{pmatrix} 6 & \omega \\ \bar{\omega} & 7 \end{pmatrix}$ or $P=[6,7,\omega]$. 
\end{example}

\subsection{The Gross lattice $\LG$}
The condition $\alpha=\alpha^{\dag}$ is equivalent to ${^t \bar{\alpha}} P = P \alpha$. By composing
the isomorphisms $\Ends(A,\lambda_0) \to NS(A) \to \Ends(A,\lambda)$, $\alpha \mapsto \lambda^{-1} \lambda_0 \alpha=P^{-1} \alpha$, a basis of $\Ends(A,\lambda)$ is given by the images by $P^{-1}$ of the basis of $\Ends(A,\lambda_0)$ consisting of
\begin{equation} \label{eq:eis}
e_1=\begin{pmatrix} 1 & 0 \\ 0 & 0 \end{pmatrix}, \, e_2=\begin{pmatrix} 0 & 0 \\ 0 & 1 \end{pmatrix},\, e_3=\begin{pmatrix} 0 & 1 \\ 1 & 0 \end{pmatrix} \, \textrm{and }e_4=\begin{pmatrix} 0 & \omega \\ \bar{\omega} & 0 \end{pmatrix}.
\end{equation}
On can find a basis for the Gross lattice $\LG$ taking the image by $\tau$ of $P^{-1} e_1,\ldots,P^{-1} e_4$ and then  using the Smith normal form algorithm.
We denote by $\ell_1,\ell_2,\ell_3$ a basis of $\LG$.

\begin{example}
For our running example, we find $$\ell_1=\begin{pmatrix}
    1 & -2\omega \\ -2+2\omega & -1 
\end{pmatrix}, \ell_2=\begin{pmatrix}
    5 & 4 \omega \\ 2 - 2 \omega & -5 
\end{pmatrix} \, \textrm{and } \ell_3=\begin{pmatrix}
    3-2 \omega & 14-4\omega \\ 8+4\omega & -3+2 \omega
\end{pmatrix}.$$
\end{example}

\subsection{Quadratic form on $L$} \label{sec:quad}  It is easy to give an explicit expression for the quadratic form $q$ for an element $\alpha \in \Ends(A,\lambda) \subset M_2(K)$. Indeed, we have that $T_{\ell}(A)=T_{\ell}(E) \oplus T_{\ell}(E)$. We see that $\Tr(\alpha)=\tr_{K/\Q} \tr(\alpha)$  where the latter is the trace of $\alpha$ seen as a matrix in $M_2(K)$. Note that, looking at the elements $P^{-1} e_i$,  we can check that the matrices in $\Ends(A,\lambda)$ have  a trace in $\Q$ so $\Tr(\alpha)=2\tr(\alpha)$. On the other hand, for any 2 by 2 matrix $M$, one has $\tr(M^2)=\tr(M)^2-2 \det M$ so we have the following equivalent expressions $$q(\alpha)=2 \tr(\alpha^2) - \tr(\alpha)^2 = \tr(\alpha)^2-4\det(\alpha).$$ 
For $\beta \in \LG$, we therefore have the very simple expression $q(\beta)/4= -\det(\beta)$.

\begin{example} \label{ex:gram163}
We can give the Gram matrix of the quadratic form $q/4$ on $\LG=\langle \ell_1,\ell_2,\ell_3 \rangle$.
$$\Gram(\LG,q/4)= (b_q/4(\ell_i,\ell_j))_{i,j \in \{1,2,3\}} = \begin{pmatrix}
    165&- 241 & 317 \\ -241 & 353 & -463 \\  317 & -463 & 613 
\end{pmatrix} \sim 
\begin{pmatrix}
    5 & -1  & 2 \\
-1  & 5  & 0 \\
 2 &  0 & 28
\end{pmatrix}.$$
\end{example}

\subsection{The Gross lattice $\bar{\LG}$ of $(\bar{A},\bar{\lambda})$} \label{sec:grossLG}
Let $(\bar{A},\bar{\lambda})/\k_{\frakp}$ be the special fiber of the stable model of $(A,\lambda)$ over $\Z_{\frakp}$. Since $\bar{A}=\bar{E}^2$ is smooth, we have that  $\Ends(A,\lambda) \hookrightarrow \Ends(\bar{A},\bar{\lambda})$. If $\bar{E}$ is ordinary, i.e. $p$ is split in $\Oc$, since $\Oc$ is the maximal order, then $\End(\bar{E}) \simeq \End(E)$, $\Ends(A,\lambda) \simeq \Ends(\bar{A},\bar{\lambda})$ and $\bar{\LG}:=\tau(\Ends(\bar{A},\bar{\lambda}))$ is isometric to $\LG$. Since $\lambda$ is indecomposable, there is no $\ell_0 \in L$ such that $q/4(\ell_0)=1$. Therefore, there is neither such an element in $\bar{\LG} \simeq \LG$ and $\frakp$ is not a prime of decomposable reduction by Theorem~\ref{th:kaniq}.

Hence, if $\frakp$ is a prime of decomposable reduction then 
 $\bar{E}$ is supersingular, i.e. $p$ is inert or ramified in $\Oc$; i.e. the Legendre symbol $\left(\frac{-d}{p}\right)=-1$ or $0$ when $p \ne 2$. 
 \begin{lemma}
Let $p$ be a prime and $-d$ be a fundamental discriminant. Assume that $d>4$ and that $p \notin \{d,d/4\}$. Then all the geometric endomorphisms of $\bar{E}$ are defined over the residue field $k_{\frakp}$ of $\Z_{\frakp}$ and we denote by $\Rc:=\End(\bar{E})$. It is a maximal order in the definite quaternion algebra  $\Bc_{p,\infty}=\Rc \otimes \Q$ of discriminant $p$. 
\end{lemma}
Observe that, in our situation, we necessarily have $d>4$, otherwise there is no indecomposable principal polarization on $E^2$. The second condition can actually be weakened: as we shall see in Theorem~\ref{th:finitep}, $p \leq d/4$ so that only $p=d/4$ may occur. In our experiments in Section~\ref{sec:datum}, we encounter this case only for $d=8$.
\begin{proof}
    If $p$ is inert, or ramified with $p \ne d$ or $d/4$, then by  Lemma~\ref{lem:decomp}, $\frakp | p$  has an inertia degree equal to 2 and therefore $k_{\frakp}= \F_{p^2}$.
    According to \cite[Prop.4.1]{waterhouse}, over $\F_{p^2}$, the Weil polynomial of $\bar{E}$ can be $(X \pm p)^2, X^2 \pm p X+p^2$ or $X^2+p^2$. In the first case, as written in \textit{loc. cit.}, all endomorphisms of $\bar{E}$ are defined over $\F_{p^2}$ and we therefore get a maximal order in a quaternion algebra over $k_{\frakp}$. The last two cases imply that $\End(\bar{E}) \otimes \Q = \Q(\sqrt{-3})$ or $\Q(\sqrt{-1})$. However since $\Oc=\End(E) \hookrightarrow \End(\bar{E})$ and $d>4$, these cases cannot occur.
\end{proof}

In practice, one would like to avoid doing explicit computations with the geometric objects. This will be possible when we will consider not a single $E$ but the full Galois orbit of $E$ in Section~\ref{sec:PDR}, which is our main interest. The following steps were therefore not implemented in our programs and we refer to them vaguely. We assume that one can compute
\begin{itemize}
    \item a smooth model of $E$ over $\Z_{\frakp}$ and its reduction $\bar{E}/\F_{p^2}$ (this is Tate's algorithm \cite{tate-alg}, and it is implemented in \texttt{Magma}); 
    \item the ring of endomorphisms $\Rc=\End(\bar{E})$ (see for instance \cite{eisentrager});
    \item  the reduction map $\End(E) \hookrightarrow \End(\bar{E})=\Rc$.    
\end{itemize}

Now, to check if $\frakp$ is a prime of decomposable reduction, one could work in principle with any basis of $\Rc$. However, for future use in Section~\ref{sec:PDR}, it will be important to choose a particular basis and we take right away this path. Since $\Oc$ is maximal, the embedding of $\Oc$ in $\Rc$ is optimal and we can extend the basis $1,\omega$ of $\Oc$ to a basis of $\Rc$, which we denote $1,\omega,r_1,r_2$ (see \cite[Exercise 8 in Chap. 10]{QAVoight}). This determines an explicit image of $P \in M_2(\Rc)$ which we will still denote $P$.

A basis of $\Ends(\bar{A},\bar{\lambda})$ is given by the union of $(P^{-1} e_i)_{1 \leq i \leq 4}$, where the $e_i$'s are defined in \eqref{eq:eis}, and the images by $P^{-1}$ of
$$e_5=\begin{pmatrix} 0 & r_1 \\ \bar{r}_1 & 0 \end{pmatrix}, \, e_6=\begin{pmatrix} 0 & r_2 \\ \bar{r}_2 & 0 \end{pmatrix}.$$
It is easy to construct a basis of $\bar{\LG} :=\tau(\Ends(\bar{A},\bar{\lambda})\})$:
\begin{lemma} \label{lem:LGbase}
The basis $\ell_1,\ell_2,\ell_3$ of the lattice $\LG \subset \bar{\LG}$ can be completed by $\ell_4=\tau(P^{-1} e_5)$ and $ \ell_5=\tau(P^{-1} e_6)$ into a basis of $\bar{\LG}$.
\end{lemma}
\begin{proof}
It is enough to check that 
\begin{eqnarray*} 5 & =& \dim \langle \tau(P^{-1}e_1),\tau(P^{-1}e_2),\tau(P^{-1}e_3),\tau(P^{-1}e_4),\tau(P^{-1}e_5),\tau(P^{-1}e_6 \rangle \\
& =& \dim \langle \tau(P^{-1}e_1),\tau(P^{-1}e_2),\tau(P^{-1}e_3),\tau(P^{-1}e_4) \rangle+\dim \langle \tau(P^{-1}e_5),\tau(P^{-1}e_6) \rangle.
\end{eqnarray*}This is because $P^{-1}$ is an isomorphism and $\tau$ has a 1-dimensional kernel.
\end{proof}

\begin{example} \label{ex:theells}
    We resume our example with $d=163$.  Let us consider for instance $p=17$ which is inert in $K=H$, so we take the prime ideal $\frakp=(17)$. We have $ \Bc_{p,\infty}=\langle 1,i,j,k\rangle$, with $i^2=-3, j^2=-17$ and $k=ij$. There are 2 supersingular elliptic curves over $\F_{p^2}$ (those with $j$-invariants $0$ and $8$), but only one has an endomorphism ring $\Rc \subset \Bc_{p,\infty}$ which contains $\Oc=\Z[\omega]$ (the one with $j=8$ since $\Rc $ does not have extra elements of norm $1$). There are only two possibilities for the image of $\omega$, namely 
    $1/2 - 4/3i - 5/4j + 5/12k$ or its conjugate. As we did not fix an identification of $\End(\bar{E})$ with $\Rc$, we cannot decide. However, it was shown in \cite[Prop.4.1]{gelin} that $(E^2,\lambda_0 P)$ is isomorphic to $(E^2,\lambda_0 \bar{P})$, so we can choose arbitrarily. We then take 
   \begin{multline*}1,\,\omega=1/2 - 4/3i - 5/4j + 5/12k,\,r_1=-1/2 + 1/2i + 1/2j - 1/2k,\\ \text{ and }r_2=-1/2 + 2/3i + 3/4j + 5/12k.\end{multline*}
 We find that a basis of $\bar{\LG}$ is given by $\ell_1,\ell_2,\ell_3$ and
    $$\ell_4=\begin{pmatrix} 
     -69 -116 \omega -168r_1  -86r_2 & 14r_1 \\
-12  -12r_1 &  -69 -118\omega -170r_1  -86r_2
     \end{pmatrix}$$ and $$\ell_5= \begin{pmatrix} 
     136+ 232\omega +334r_1+ 170r_2 &  14 r_2 \\
-12 -12r_2 & 136 +230\omega+ 334r_1 +168r_2
     \end{pmatrix}.$$
 \end{example}

\subsection{The quadratic form on $\bar{\LG}$}
Since the quadratic form $q_{(\bar{A},\bar{\lambda})}$ coincides with $q$ on the sublattice $\Ends(A,\lambda)$, we still denote it by $q$ and by $b_q$ the associated bilinear form. However, $P^{-1} e_5$ or $P^{-1} e_6$ do not have necessarily a trace in $\Q$, so the explicit form of $q$ on $\Ends(\bar{A},\bar{\lambda})$ cannot be simplified as in Section~\ref{sec:quad}. We have for $\alpha \in \Ends(\bar{A},\bar{\lambda})$: $$q(\alpha)= \trd(\alpha^2)-\frac{1}{4} \trd(\alpha)^2,$$
where $\trd(\alpha)=\tr_{\Rc/\Q}(\tr(\alpha))$. Hence, for $\beta \in \bar{\LG}$, we have $q(\beta)=\trd(\beta^2)$.

\begin{example}
We continue with our example and the prime $\frakp=(17)$.    The Gram matrix of $\bar{\LG}$ in the basis $\ell_1,\ldots,\ell_5$ with respect to $q/4$ is equal to
$$\textrm{Gram}(\bar{\LG},q/4)= \begin{pmatrix}
  165 & -241&   317 & -598 & -247 \\
-241 &  353 & -463 &  874 &  361 \\
  317 & -463 &  613 &-1152 & -477 \\
 -598  & 874 & -1152  & 2188 &  862 \\
 -247 &  361&  -477  & 862 &  441
\end{pmatrix}.$$
It contains in its upper corner the $3\times3$ matrix $\Gram(\LG)$ in Example~\ref{ex:gram163}.
\end{example}

\subsection{An algorithm to test the decomposability at a prime $\frakp$}
Having all these explicit computations at hand, Algorithm~\ref{algo:primereduction} determines if a given $\frakp$ is a prime of decomposable reduction by looking for an element of norm 1 in the lattice $\bar{\LG}$ with positive definite Gram matrix of $q/4$. If this element exists, it is unique up to a sign.
 \begin{proposition} \label{prop:uniquel0}
     If there is an element $\ell_0 \in \bar{\LG}$ with $(q/4)(\ell_0)=1$ then $\ell_0$ is unique up to a sign.
 \end{proposition}
 \begin{proof} As we have seen in the proof of Theorem~\ref{th:kaniq}, an endomorphism $\ell_0 \in \bar{\LG}$ with $q/4(\ell_0)=1$ leads to symmetric idempotents $e_1=(\ell_0+ 1_{\End(\bar{A}})/2$ and $e_2=1_{\End(\bar{A})}-e_1$ and to an isomorphism of principally polarised abelian surfaces $(E^2, \lambda)\rightarrow (\im(e_1)\times \im(e_2), \lambda_0)$ where $\lambda_0$ is the product polarization $1 \times 1$. Reciprocally, given such an isomorphism $\phi :\,(E^2, \lambda)\rightarrow (E_1\times E_2, \lambda_0)$, let us define $e_1=\phi^{-1} \circ \pi \circ \phi$ where $\pi \in \End(E_1 \times E_2)$ is defined by $(x,y) \to (x,0)$ and $e_2=1-e_1$. We see that $e_1^2=e_1$ and $e_1^{\dag} = \lambda^{-1} \circ \hat{e}_1 \circ \lambda =e_1$ since $\lambda = \hat{\phi} \lambda_0 \phi$ and $\lambda_0^{-1} \pi \lambda_0=\pi$. Let us define $\ell_0=\tau(e_1) \in \bar{\LG}$. Using Remark~\ref{rem:minpol}, since $e_1$ is symmetric and $e_1^2-e_1=0$ we see that $\Tr(e_1^2)=\Tr(e_1)=2$ and therefore $q/4(\ell_0)=1$. 
 In \cite[Prop. 3.3]{gelin}, it is proved that the isomorphism between $(E^2,\lambda)$ and $(E_1 \times E_2, 1 \times 1)$ is unique up to  interchanging $E_1$ and $E_2$, and up to isomorphism for each elliptic curve. The latter does not affect the $e_i$ and exchanging $E_1,E_2$ only exchanges the idempotents $\{e_1,e_2\}$ which only multiplies $\ell_0$ by $-1$. The result then follows. 
 \end{proof}

\begin{algorithm}
    \caption{Is $\frakp$ a prime of decomposable reduction for $(E^2,\lambda)$?}
\label{algo:primereduction}
\begin{algorithmic}[1]
    \Require A maximal order $\Oc=\Z[\omega]$ of discriminant $-d<0$ in an imaginary quadratic field $K$; a matrix $P \in M_2(\Oc)$ representing an indecomposable principal polarization $\lambda$ on $E^2$ where $E$ is an elliptic curve with CM by $\Oc$ defined over the Hilbert class field $H$ of $K$; $\frakp$ a prime ideal in $H$ over a prime $p$; $\Z_{\frakp}$ the ring of integers of $H_{\frakp}$. 
    \Ensure Return true if the stable model of $(E^2,\lambda_0 P)$ over $\Z_{\frakp}$ has a special fiber which is decomposable. 
    \State Check if $p$ is inert or if $p$ is ramified in $\Oc$ and $p \notin \{d/4,d\}$. If not return false.
    \State Compute the stable model $\Ec$ of $E$ over $\Z_{\frakp}$ and its reduction $\bar{E}$ over $\F_{p^2} \subset k_{\frakp}$.
    \State Compute the endomorphism ring  $\Rc:=\End(\bar{E})$ in $\Bc_{p,\infty}$.
    \State Compute the image of  $\omega \in \End(E)$ in $\Rc$ and a basis of $\Rc$ as $1,\omega,r_1,r_2$.
    \State Compute the image of $P$ in $M_2(\Rc)$. We still denote it $P$.
    \State Compute a basis $\ell_1,\ell_2,\ell_3$ of $\LG=\tau(\Ends(E^2,\lambda))$ by taking the images by $\tau$ of $P^{-1} e_i$ for $1 \leq i \leq 4$ where the $e_i$s are
      $$\begin{pmatrix} 1 & 0 \\ 0 & 0 \end{pmatrix}, \, \begin{pmatrix} 0 & 0 \\ 0 & 1 \end{pmatrix},\, \begin{pmatrix} 0 & 1 \\ 1 & 0 \end{pmatrix},\, 
      \begin{pmatrix} 0 & \omega \\ \bar{\omega} & 0 \end{pmatrix}.$$
    \State Compute a basis of $\bar{\LG}=\tau(\Ends(\bar{E}^2,\bar{\lambda}))$ by adding to $\ell_1,\ell_2,\ell_3$ the matrices  
 $$   \ell_4=   \tau\left(P^{-1} \begin{pmatrix} 0 & r_1 \\ \bar{r_i} & 0 \end{pmatrix}\right), \, 
  \ell_5=  \tau\left(P^{-1} \begin{pmatrix} 0 & r_2 \\ \bar{r_2} & 0 \end{pmatrix}\right).$$ 
  \State Compute the positive definite Gram matrix of $\bar{\LG}$ with respect to $q/4$ where $q/4(\beta)=\trd(\beta^2)$.
  \State \Return true if and only if the lattice $(\bar{L},q/4)$ contains an element $\ell_0$ of norm $1$.
\end{algorithmic}
\end{algorithm}

\begin{remark}\label{cased/4} Notice the special case $p=d/4$ (as said, $p=d$ cannot occur). In our experiments (for $d <200$), this case only shows up for $d=8$. In such a case, the residue field $k_{\frakp}=\F_p$, so  $\End(\bar{E})=\Oc$ and the polarization $\bar{\lambda}$ on $\bar{E}^2$ remains indecomposable. However, over $\F_{p^2}$, the endomorphism ring is $\Rc$ and there the polarization $\bar{\lambda}$ may decompose. This is what does happen for $d=8$ and $p=2$. 
\end{remark}

\begin{example} \label{ex:prime17}
Let us finish our running example.  
The endomorphism 
$$\ell_0=8 \ell_1-2\ell_2-\ell_3-2\ell_4-\ell_5=\begin{pmatrix}
-3 +2 \omega  + 2r_1  + 2r_2 & -14 -  20\omega-28r_1 -14r_2 \\
8 +16 \omega + 24 r_1+ 12r_2 &  7 +  4\omega +6 r_1 +4r_2
\end{pmatrix}$$ satisfies $q/4(\ell_0)=1$, so $\frakp=(17)$ is a prime of decomposable reduction, which is coherent with \eqref{eq:villegas163} as we will see in Section~\ref{sec:disc}. 
We defines the idempotent 
$$e_1=(\ell_0+\id_2)/2=
\begin{pmatrix}
-1 + \omega  + r_1  + r_2 & -7 -  10\omega-14r_1 -7r_2 \\
4 +8 \omega + 12 r_1+ 6r_2 &  4 +  2\omega +3 r_1 +2r_2
\end{pmatrix}.$$
Let us define now
$$
Q=\begin{pmatrix} -3 -7\omega-10r_1 -5r_2 & -1-2\omega-3r_1-2r_2\\  -4 -8\omega-12r_1-6r_2 & -3-2\omega-3r_1-2r_2
\end{pmatrix},
$$
where the first row is a linear combination of the rows of $e_1$ with entries of minimal norm and similarly for the second row with the rows of  $e_2=(\ell_0-\id_2)/2$. One has ${{Q}^\vee} Q=2P$ where $\vee$ denotes the dual endomorphism after identification of $\bar{E}$ and its dual. The matrix 
$$
{{Q}^\vee}=\begin{pmatrix} {Q}_{11}^\vee & {Q}_{21}^\vee\\   {Q}^\vee_{12} & {Q}^\vee_{22}
\end{pmatrix}=\begin{pmatrix} t_{11}s_1 & t_{12}s_2\\  t_{21}s_1 & t_{22}s_2
\end{pmatrix}=T\circ\operatorname{diag}(s_1,s_2),
$$
where $s_i:\,\bar{E}\rightarrow \bar{E}_i$ is a $2$-isogeny: ${Q}^\vee_{11} $ has norm $4$ and  $Q^\vee_{12}$ has norm $6$. By checking if $({Q}^\vee_{11})^2$, $({Q}_{12}^\vee)^2$ and ${Q}^\vee_{11}+{Q}^\vee_{12}$ factor through $[2]$, we conclude that they share a 2-torsion point in their kernel;  the endomorphisms ${Q}^\vee_{21}$ and ${Q}^\vee_{22}$ have both norm $8$ and ${Q}^\vee_{21}$ factors though $[2]$, so the same applies. Hence, $T:\,(\bar{E}_1\times \bar{E}_2,1\times 1)\rightarrow (\bar{E}^2,\lambda)$ is an isomorphism of principally polarised abelian surfaces:

\begin{center}
\begin{tikzcd}
\bar{E}^2  \arrow[dd,"2\lambda"] \arrow[rightarrow]{rr}{T^\vee} \arrow[rrrr, bend left=35, "Q"] &
  & \bar{E}_1\times \bar{E}_2 \arrow[rightarrow]{rr}{diag(s_1^\vee,s_2^\vee)}  \arrow[dd,"2(1\times1)"] 
 & & \bar{E}^2  \arrow[dd,"1\times1"]
      \\
      &{\circlearrowleft} & & \circlearrowleft & \\
 \bar{E}^2  \arrow[leftarrow]{rr}{T}  &
  & \bar{E}_1\times \bar{E}_2 \arrow[leftarrow]{rr}{diag(s_1,s_2)}  
 & & \bar{E}^2   \arrow[llll, bend right=-20, "Q^\vee"]
\end{tikzcd}
\end{center}

\end{example}

\section{Bounds on the primes of decomposable reduction}
We have seen in Section~\ref{sec:algoprime} that given a prime ideal $\frakp$ over $p$, we can determine whether $(E^2,\lambda)$ has decomposable reduction at $\frakp$. In the following section, we give an upper bound on the primes $p$ which can appear. Combined with Algorithm~\ref{algo:primereduction}, this will allow to find all primes $\frakp$ of decomposable reduction for a given $(E^2,\lambda)$. We resume with the notation of the previous section.

\subsection{The discriminant of $\Gram(\LK,q)$}
Although the basis $\ell_1,\ell_2,\ell_3$ of $\LG$ cannot be given by a close formula, it is possible to compute the determinant of its Gram matrix $\Gram(\LG,q/4)$ with respect to the quadratic form $q/4$. It is equal to the one of the isometric lattice $(\LK,q)$ for which one has the following result in the literature.

\begin{lemma} \label{lem:detL}
The determinant of the Gram matrix $\Gram(\LK,q)$ of $\LK$ with respect to the quadratic form $q$ is equal to $4d$.
\end{lemma}
\begin{proof}
We use \cite[Prop.9]{kanihumbert} which shows\footnote{Actually Kani's proof would lead to $32 d$ since he defines the bilinear form attached to $q$ without the multiplicative factor $\frac{1}{2}$. Therefore, with his notation, we must replace in his proof $ \beta_q(y_1,y_2)=4\beta(y_1,y_2)$ by $\beta_q(y_1,y_2)=2 \beta(y_1,y_2)$.} that $$\det \Gram(\LK,q)= \frac{1}{2} (-2)^{\rk(\NS(A))-1} \det(\NS(A),(.))= -4 \det(\NS(A),(.)).$$
It remains to compute $\det(\NS(A),(.))$ using \cite[Cor.24]{kanihumbert}:
$$\det(\NS(E^2),(.))= (-1)^{\rk(\NS(A))-1} \cdot \det(\End(E),\beta),$$ where $\beta(x,y)=\deg(x+y)-\deg(x)-\deg(y)$ for $x,y \in \End(E)$. We perform an explicit computation on the basis $\{1,\omega\}$ of $\End(E)$ and we see that the last determinant is equal to 
$$\det \begin{pmatrix} 2 & \tr(\omega) \\ \tr(\omega) & \omega \bar{\omega} \end{pmatrix} = d.$$
\end{proof}

\subsection{A first bound}

Let $\Pc$ be the unique two-sided prime ideal of $\Rc$ over $p$ with residual degree $2$ (see \cite[Th.18.1.3]{QAVoight}). Mimicking \cite[Lem.5.7]{kudla}, we denote by $\res$ the restriction map which goes from $\End(\bar{A})=M_2(\Rc)$ to $M_2(\Rc/\Pc)=M_2(\F_{p^2})$. This could be interpreted as the restriction map to the kernel of $p$-power Frobenius but we simply look at it at the level of matrices. 

We will  reduce our quadratic form $q$ modulo a prime $p$. As we want to include $p=2$, it is easier to switch from the description in terms of Gross lattices and $q/4$ to Kani lattices and $q$. Remember that the two lattices are isometric, in particular their Gram matrices are equivalent.
 \begin{lemma} \label{lem:red}
Let $\bar{L}^{Kani}: =\{\Ends(\bar{A},\bar{\lambda})\}/\Z \cdot 1_{\End(\bar{A})}$. The $\F_p$-vector space $\res(\bar{L}^{Kani})$ has dimension $3$.
\end{lemma}
\begin{proof}
The dimension of the $\F_p$-vector space $\res(\bar{L}^{Kani})$ is one less than the dimension of $\res(\Ends(\bar{A},\bar{\lambda}))$. Since the matrix $P$ has integral coefficients and determinant $1$, the dimension of the latter is the same as the dimension of $\res(\Ends(\bar{A},{\lambda}_0))$. Since $\Rc/\Pc=\F_{p^2}$, it is then easy to see that $\res(e_1),\ldots,\res(e_6)$ generate a $\F_p$-vector space of dimension 4 inside $M_2(\F_{p^2})$.
\end{proof}

 According to Lemma~\ref{cor:bqintegral},
$b_q$ is integral so we can reduce it modulo $p$ and define a quadratic form $\bar{b}_q$ over $\F_p$. The trace operations on matrices and on elements of $\Rc$ are compatible with the reduction modulo $\Pc$ so we get 
\begin{equation} \label{eq:qqtilde}
    b_q(\alpha_1,\alpha_2) \pmod{p} = \bar{b}_q(\res(\alpha_1),\res(\alpha_2)).
\end{equation}

We can now formulate the following necessary condition for primes of decomposable reduction.
\begin{proposition} \label{prop:cond1}
If $(A,\lambda)$ has decomposable reduction at a prime $\frakp$ over $p$ then there exists $(x,y,z) \in \Z^3$ such that the matrix 
\begin{equation} \label{eq:Tmatrix}
T =\begin{pmatrix} \Gram(\LK) & \begin{array}{c}x \\ y \\ z\end{array}\\\begin{array}{ccc}x & y & z\end{array} & 1\end{pmatrix}
\end{equation}
is positive definite (equivalently has positive determinant) and $p | \det(T)$.
\end{proposition}
\begin{proof}
 By Theorem~\ref{th:kaniq}, there exists $\ell_0 \in \bar{L}^{Kani}=\Ends(\bar{A},\bar{\lambda})/\Z \cdot 1_{\End(\bar{A})}$ with $q(\ell_0)=1$. However, since $\lambda$ is indecomposable, there is no endomorphism in $\LK$ of norm $1$ so $\ell_0 \notin \LK$.  The vector $\ell_0$ cannot even be in $\LK \otimes \Q$. Indeed $\bar{L}^{Kani} \cap (\LK\otimes \Q) = \LK$ since $e_5,e_6$ are $\Q$-linearly independent of $e_1,e_2,e_3,e_4$.
 
The Gram matrix $T$ of the lattice $\LK \oplus \langle \ell_0 \rangle$ is therefore  positive definite. On the other hand, by Lemma~\ref{lem:red}, we know that the $\F_p$-span of $\res(\LK \oplus \langle \ell_0 \rangle)$ has dimension at most $3$. This implies that the Gram matrix of $\res(\LK \oplus \langle \ell_0 \rangle)$ for the quadratic form $\bar{b}_q$ has determinant $0$. By \eqref{eq:qqtilde}, we see that $T \pmod{p}=\Gram(\LK \oplus \langle \ell_0 \rangle,b_q) \pmod{p} = \Gram(\res(\LK \oplus \langle \ell_0 \rangle),\bar{b}_q)$ and this allows to conclude that  $p | \det(T)$.
\end{proof}

This proposition leads to a finite list of possible $p$'s using the following folklore lemma.
\begin{lemma} \label{lem:Qmat}
    Let $T=\begin{pmatrix} S & m \\ {^t m} & 1 \end{pmatrix} \in M_n(\R)$ with $S>0$, $m \in \R^{n-1}$ and $n \geq 2$. Then $T>0$ if and only if $\det T= \det S - Q(m)>0$ where $Q$ is the positive definite quadratic form associated to the adjugate matrix of $S$.
\end{lemma}
\begin{proof}
   We have $T>0$ if and only if $\det(T)>0$. Now 
    \begin{eqnarray*}
    \det \left(\begin{pmatrix} \textrm{adj}(S) & 0 \\ 0 & 1 \end{pmatrix} \cdot T\right) &= &
    \det \begin{pmatrix} \det(S) \textrm{Id}_{n-1} & \textrm{adj}(S) m \\ {^t m} & 1 \end{pmatrix} \\
    &= & \det \begin{pmatrix} \det(S) & 0 & 0 & * \\ 0  & \ddots & 0 & * \\ 0 & 0 & \det(S) & * \\ 0 & 0 & 0 & 1-Q(m)/\det(S) \end{pmatrix} \\
    &=&  \det(S)^{n-2} (\det(S)-Q(m))
    \end{eqnarray*}
by cancelling the $m$'s coefficients of the last row. Since $\det \textrm{adj}(S) = \det(S)^{n-2}$, we get the announced result.
\end{proof}

\begin{proposition}[intermediary bound]
    If $A$ has decomposable reduction at a prime $\frakp$ over $p$ then $p < 4d$.
\end{proposition}
\begin{proof}
     If $A$ has decomposable reduction at $\frakp$  then let $T \in M_4(\Z)$ be one of the matrix from Proposition~\ref{prop:cond1}. By  Lemma~\ref{lem:detL}, we know that $p  | \det(T) = 4d-Q(m) \leq 4d$ and hence, the prime $p$ must be smaller than $4d$.
\end{proof}

\subsection{An improved bound}
A formal computation with a software using Corollary~\ref{cor:bqintegral}~\eqref{cor:cong} shows the following:

\begin{lemma} \label{lem:Qcong}
Let $(A,\lambda)$ be an abelian surface over an algebraically closed field $k$ and $\ell_1, \ell_2,\ell_3$ and $\ell_0$ in $\Ends(A,\lambda)$ with $q(\ell_0) \equiv 1 \pmod{4}$. 
Then $$\det \Gram(\langle \ell_1,\ell_2,\ell_3,\ell_0 \rangle,q) \equiv 0 \pmod{16}.$$
\end{lemma}

\begin{theorem}[final bound] \label{th:finitep}  If $(A,\lambda)$ has decomposable reduction at a prime $\frakp$  over $p$ then $p \leq d/4$ and $p$ is inert or ramified in $\Oc$.
\end{theorem}
\begin{proof}
We only have to show that $p \leq d/4$ as we have already observed that $p$ must be inert or ramified on $\Oc$. Using Proposition~\ref{prop:cond1} and Lemma~\ref{lem:Qmat} and the notation there, we have $$p | det(T)=det(S)-Q(m)=4d-Q(m).$$ By Lemma~\ref{lem:Qcong}, if $p\ne 2$, $p | (4d-Q(m))/16 \leq  4d/16=d/4$. As $d \geq 8$, we can include $p=2$ for free.
\end{proof}

\begin{example} \label{ex:sieving}
    Let us resume with our running example with $d=163$. If we simply use the critera of Theorem~\ref{th:finitep}, we get the set
    $$\P=\{2,3,5,7,11,13,17,19,23,29,31,37\}$$
    of possible primes $p$ for which there may exists $\frakp | p$ of decomposable reduction.
    If we run precisely over the matrices $T>0$ in \eqref{eq:Tmatrix} with the extra-condition $\det T \equiv 0 \pmod{16}$, we find 19 matrices $T$ (we actually only consider one of the two equal solutions with $m$ or $-m$) and the determinants of $T$ factorise as 16 times    
    $$2^5,31,3^3,23,2^2 \cdot 5,2^2 \cdot 5,19,2 \cdot 3^2,17,17,13,11,11,7,5,5,3,2,2.$$    
    So the set of primes reduces to $\P=\{2,3,5,7,11,13,17,19,23,31\}$.

    It is actually even more interesting to sieve the matrices $T$ by adding the conditions from Corollary~\ref{cor:bqintegral}~\eqref{cor:cong} expecting an $\ell_0 \in \bar{L}^{Kani}$ with $q(\ell_0)=1$:
    \begin{equation} \label{eq:sieve}
    x^2 \equiv q(\ell_1) \pmod{4}, \, y^2 \equiv q(\ell_2) \pmod{4} \; \textrm{and } z^2 \equiv q(\ell_3) \pmod{4}.    
    \end{equation}
    Doing so, one only finds 8 possible matrices $T$ whose determinants factorise as 16 times
    $$2, 3^3,  2^2 \cdot 5, 7, 11,17,  19, 23.$$
    All the prime factors $2,3,5,7,17,19$ and $23$ in this list are inert in $\Oc$. This is a coincidence: for instance, for $d=83$ and the polarization $P=[2,11,\omega]$, we find that the determinants of the possible $T$ factorise as $16$ times $3 \cdot 5,2 \cdot 7,2^3,5$
    but $7$ is not inert in $\Oc$.
\end{example}

\begin{remark} \label{rem:trueT}
    The links between the set of matrices $T$ above at ramified or inert primes and the set of Gram matrices of $L \oplus \ell_0$ that one could build from true solutions $\ell_0 \in \Ends(\bar{A},\bar{\lambda})$ with $q/4(\ell_0)=1$ are still mysterious, see Section~\ref{sec:observations}. But one can show that the inclusion is strict. For instance, for $d=163$ and  the polarization $P=[3,14,\omega]$, we get two matrices $T$ with determinant divisible by 3 (namely $2^4 \cdot 3^3$ and $2^5 \cdot 3^2)$, but only the first one is equivalent to a true solution matrix.
\end{remark}

\begin{remark}\label{rem:boundg3} The arguments in this section together with Remark \ref{rem:g3} will likely produce a bound for the primes of geometrically bad reduction of genus 3 curves with complex multiplication in the non-primitive setting. This will complement the works in the primitive setting done in \cite{WINE1,WINE1-2,WINE2-2}.
\end{remark}

\section{Primes of potentially decomposable reduction} \label{sec:PDR}
Let $\Oc$ be a maximal order in an imaginary quadratic field $K$ and $P \in M_2(\Oc)$ be a hermitian matrix of determinant $1$ which induces an indecomposable polarization $\lambda=\lambda_0 P$ on the square of any elliptic curve with CM by $\Oc$. Let $A_2$ be the moduli space of principally polarised abelian surfaces.

Instead of looking at a particular elliptic curve  and prime $\frakp$ over $p$, it is more natural to look at all Galois conjugates at the same time, i.e. all elliptic curves with complex multiplication by $\Oc$ or all primes $\frakp$ over $p$. This motivates the following definition.
\begin{definition}
    We say that a prime $p$ is a prime of \emph{potential decomposable reduction} (PDR) if there exists a prime $\frakp$ over $p$ in the Hilbert class field $H$ of $K$ and an elliptic curve $E/H$ with CM by $\Oc$  such that $(E^2,\lambda)$ has geometrically decomposable reduction at $\frakp$.
\end{definition}
\begin{remark}
    We need to add ``geometrically'' to include the possible cases $p=d/4$ which appears, for instance, when $d=8$. Indeed, only for such cases, the reduction of $(E^2,\lambda)$ is not decomposable over $k_{\frakp}=\F_p$ but possibly over $\F_{p^2}$.
\end{remark}

Since the elliptic curves $E$ with CM by $\Oc$ form a Galois orbit under $\Gal(H/K)$, it is redundant to go trough $\frakp$ and $E$ simultaneously. In this section, we therefore fix a prime $\frakp$  over each $p$. All the results will be independent of the choice of these $\frakp$.

Since $P \in M_2(K)$, the action of $\Gal(H/K)$ does not affect the polarization $\lambda_0 P$ and we can consider the $K$-cycle given by the Galois orbit of $(E^2,\lambda_0 P)$.  Actually, as it was shown in the proof of \cite[Prop.4.1]{gelin},  $(E^2,\lambda_0 \bar{P})$ is isomorphic to $(E^2,\lambda_0 P)$, so by taking the orbit under $\Gal(H/\Q)$, we get a $\Q$-cycle where each $(E^2,\lambda_0 P)$ is counted twice. We can  therefore consider the $\Q$-cycle
$$\textrm{CM}(\Oc,P):=\bigsqcup_{\sigma \in \Gal(H/K)} ((E^{\sigma})^2,\lambda_0 P) \subset A_2.$$
Notice that this cycle is still not necessarily reduced: an extreme case is the one studied in \cite{gelin} where all $ ((E^{\sigma})^2,\lambda_0 P)$ are isomorphic, hence the support of the cycle is a single point in the moduli space.

\subsection{Arithmetic intersection} 

The formulation above brings us closer to an arithmetic intersection problem. Let $\Ac_2$ be the moduli stack of principally polarised abelian schemes of relative dimension $2$. Let $\Gc_1$ be the locus in $\Ac_2$ of decomposable principally polarised abelian schemes. 

\begin{definition}
    Let $p$ be a prime  and $\frakp | p$ a prime ideal in $H$. Consider the $\Z_p$-cycle of $\Ac_2$
    $$\mathcal{CM}(\Oc,P,p) := \bigsqcup (\Ec_{\sigma}^2,\lambda_0 P)$$
    where, for $\sigma \in \Gal(H/K)$, $\Ec_{\sigma}$ is the smooth model of $E^{\sigma}$  over the integer ring $\Z_{\frakp}$ of $H_{\frakp}$.    
\end{definition}

Since our polarization $\lambda_0 P$ is indecomposble, this intersection is proper and we can define an arithmetic intersection multiplicity.
\begin{proposition}
    A prime $p$ is a prime of PDR if and only if the arithmetic intersection multiplicity $e_{\Oc,P,p}:=(\Gc_1 . \mathcal{CM}(\Oc,P,p)) \ne 0$.
\end{proposition}

\subsection{Support of the intersection} \label{sec:support}
Let $p$ be a prime of PDR. The points $x$ in the support of the intersection of $\Gc_1$ with  $\mathcal{CM}(\Oc,P,p)$ correspond to $(\bar{E}^2,\bar{\lambda})$ where 
\begin{enumerate}
    \item  $\bar{E}$ is a supersingular elliptic curve which is the reduction  of the smooth model of an elliptic curve $E$ with CM by $\Oc$ modulo $\frakp$;
    \item the polarization $\bar{\lambda}$ is the reduction of the polarization $\lambda$;
    \item The polarization $\bar{\lambda}$ decomposes over $\F_{p^2}$. Note that asking $\F_{p^2}$ and not $k_{\frakp}$ allows to integrate the special cases $p=d/4$.
    \end{enumerate}
As we have seen in Proposition~\ref{prop:uniquel0}, this set is in bijection with the $\pm \ell_0 \in \bar{\LG}$ such that $q/4(\ell_0)=1$. The lattice $\bar{\LG}$ depends only on the embedding of $P$ in $M_2(\End(\bar{E}))$ which in turn is induced by the natural embedding $\iota: \Oc=\End(E) \hookrightarrow \Rc=\End(\bar{E})$ up to the automorphisms of $\bar{E}$, i.e. the elements of $\Rc^{\times}$. We therefore need to be able to go through these embeddings. We will proceed as follows.

For $p$ inert or ramified in $\Oc$, let us consider the reduction map modulo $\frakp$ $$\rho:\,\bigcup_{\sigma \in \Gal(H/K)} E^{\sigma} \rightarrow \{\textrm{supersingular elliptic curves over} \; \F_{p^2}\}.$$ We are therefore interested in going through the elements of
$$\textrm{SS}_{\Oc}(\rho) =\bigcup_{\sigma \in \Gal(H/K)} \{(\rho(E^{\sigma}),\iota) : \iota:\Oc=\End(E^{\sigma})\hookrightarrow \End(\rho(E))\}_{/\simeq},$$
where the pairs $(\bar{E},\iota)$ and $(\bar{E}',\iota')$ are equivalent if there exists an isomorphism $\phi:\,\bar{E}\rightarrow \bar{E}'$ such that $\iota'=\phi\iota\phi^{-1}$. 
In order to do so, let us introduce the bigger set $$\textrm{SS}_{\Oc}(p)=\{(\bar{E},\iota):\,\bar{E}/{\F}_{p^2}\text{ supersingular},\, \iota:\Oc\hookrightarrow \End(\bar{E})\}_{/\simeq}.$$ 
We denote $\textrm{SS}_{\Oc}(\bar{\rho}):=\{(\bar{E},\sigma \circ \iota) :\,(\bar{E},\iota)\in \textrm{SS}_{\Oc}(\rho)\}$ where $\sigma$ is the Frobenius automorphism on $\bar{E}$. 
We finally let $\textrm{Emb}_{\Oc}(p):=\cup_{\Rc \subset \Bc_{p,\infty}} \{\iota: \Oc \hookrightarrow \Rc\}_{/\Rc^{\times}}$ be the set of embeddings of $\Oc$ in all (conjugacy classes of) maximal orders $\Rc$ of $\Bc_{p,\infty}$ up to conjugation by the units of $\Rc$. This is the set which can easily be implemented without any reference to the elliptic curves. We  have the following diagram:

\begin{center}
\begin{tikzcd}
 \textrm{SS}_{\Oc}(\rho) \arrow[rd,hookrightarrow]
  &   &  \textrm{Emb}_{\Oc}(p)
    \\
  & \textrm{SS}_{\Oc}(p) \arrow[ur,twoheadrightarrow,"\pi"] &     
\end{tikzcd}
\end{center}
where $\pi$ forgets about the elliptic curve $\bar{E}$.
In \cite[Section 5.3.3]{ColoPhD}, it is proved the following:
\begin{itemize}
\item if $p$ is inert then $\textrm{SS}_{\Oc}(p)=\textrm{SS}_{\Oc}(\rho)\sqcup \textrm{SS}_{\Oc}(\bar{\rho})$; 
\item if $p$ is ramified then $\textrm{SS}_{\Oc}(p)=\textrm{SS}_{\Oc}(\rho)=\textrm{SS}_{\Oc}(\bar{\rho})$.
\end{itemize}
This can  be seen as a consequence of \cite[Prop.~2.7]{grosszagier} after the theory of Serre and Tate on deformation theory of $p$-divisible groups and Lubin-Tate groups.

On the other hand, by \cite[Lem.42.4.1]{QAVoight}, we know that the map $\pi$ is surjective with one or two preimages, the second case happening if and only if $j(\bar{E}) \in \F_{p^2} \setminus \F_p$, if and only if the unique two-side ideal $\Pc$ of $\Rc$ above $p$ is not principal. This gives the following recipe to run over the elements of $\textrm{SS}_{\Oc}(\rho)$.
\begin{lemma} \label{lem:emb}
    The embeddings $\iota$ of $\textrm{SS}_{\Oc}(\rho)$ are
    \begin{itemize}
    \item if $p$ is inert and $\Pc$ is principal (resp not principal): half the elements (resp. all the elements) of $\textrm{Emb}_{\Oc}(p)$;
    \item if $p$ is ramified and $\Pc$ is principal (resp. not principal): the elements (resp. twice the elements) of $\textrm{Emb}_{\Oc}(p)$.
    \end{itemize}
\end{lemma}

In the inert case, Lemma~\ref{lem:emb} does not seem sufficiently precise to run our algorithm as we would need to identify which half is the right one. However, since $p$ is inert, $k_{\frakp}=\F_{p^2}$ and since $\Pc$ is principal, $\bar{E}$ is defined over $\F_p$. Hence, the Frobenius automorphism $\sigma$ is the Galois involution of $\F_{p^2}/\F_p$ which is induced by the complex conjugation on $K$. Now, as already mentioned $P$ and $\bar{P}$ are equivalent hence the corresponding lattices $\bar{L}$ are isometric and the Gram matrices we get using an embedding $\iota$ or $\sigma \circ \iota$ are therefore equivalent. Thus, we can run through  all embeddings and  multiple or divide by $2$ at the end of the counting. To summarise, in terms of the embeddings $\textrm{Emb}_{\Oc}(p)$, we can check if a prime $p$ is a prime of PDR by 
\begin{enumerate}
    \item Running over all conjugacy classes of maximal orders $\Rc$ in $\Bc_{p,\infty}$;
    \item for each of them, compute the embeddings $\iota: \Oc \hookrightarrow \Rc$, up to conjugation by $\Rc^{\times}$;
    \item determine the images of $P \in M_2(\Rc)$ by the $\iota$'s;
    \item for such an $\Rc$ and $P$, compute the lattice $\bar{\LG}$ and its Gram matrix with respect to  $q/4$ as in Section~\ref{sec:grossLG};
    \item Return that $p$ is a prime of PDR if and only if at least one of these lattices has a vector $\ell_0$ of norm $1$.
\end{enumerate}

\subsection{Computation of the intersection multiplicity} \label{sec:exp}
We use the framework of \cite{kudla} with some words of caution. Their work is  more general and slightly distinct. We deal here with the case $B=M_2(\Q)$ (with their notation). Moreover, they do not consider the lattice $L$  but the over-lattice $L_0$ of special endomorphisms as explained in \cite[Rem.2.2.2]{shankarred}. They also exclude the prime $p=2$, although this seems like a technical condition only used when computing the Gross-Keating invariants (see below). Nevertheless, the devissage of the local multiplicity of intersection worked out in \cite[p.733]{kudla} is identical. In practice, we have seen that each unique $\pm \ell_0 \in \bar{\LG}$ such that $(q/4)(\ell_0)=1$ corresponds to a point $x \in \Gc_1 \cap \mathcal{CM}(\Oc,P,p)$. As explained in \cite[p.733]{kudla}, the length $e_{\ell_0}$ of the local Artin rings $\Oc_{\Gc_1\cap \mathcal{CM}(\Oc,P,p),x}$ can be computed as follows. Let  $Q_{\ell_0}$ be the $3 \times 3$ Gram matrix of the lattice ${\ell_0}^{\perp} \subset \langle \ell_1,\ell_2,\ell_3,\ell_0 \rangle$. One then computes the Gross-Keating invariants $(a_1,a_2,a_3)$ of $Q_{\ell_0}$. When $p \ne 2$, this can be done by computing the Jordan  decomposition of $Q_{\ell_0}$ over $\Z_p$ using \texttt{Magma}. One then get 
\begin{align} \label{eq:GK}
e_{\ell_0} & =  \sum_{i=0}^{a_1-1} (i+1) (a_1+a_2+a_3-3i) p^i + \sum_{i=a_1}^{(a_1+a_2-2)/2} (a_1+1)(2a_1+a_2+a_3-4i) p^i  \nonumber\\
&  +\frac{1}{2}(a_1+1)(a_3-a_2+1) p^{(a_1+a_2)/2} \; \; \textrm{if } a_1+a_2 \; \textrm{is even}, \\
e_{\ell_0} & = \sum_{i=0}^{a_1-1} (i+1) (a_1+a_2+a_3-3i) p^i + \sum_{i=a_1}^{(a_1+a_2-1)/2} (a_1+1)(2a_1+a_2+a_3-4i) p^i \; \; \textrm{if } a_1+a_2 \; \textrm{is odd}. \nonumber
\end{align}
When $p=2$, the computation of Gross-Keating invariants is more involved (see for instance \cite{bouwargos}). A full implementation exists in \texttt{Mathematica} \cite{leeGK}. We did a partial implementation in \texttt{Magma} of the cases which do occur for the ternary quadratic form $Q_{\ell_0}$ for all cases $d<200$. We do not know if this is enough for any
$d$: if not, the present  implementation 
returns an error. We plan to have a full translation of the \textit{loc. cit.} algorithm in the future. 
Now, following \cite[Chapter 2]{KRY} and \cite[Section 2]{yang2010}:
$$e_{\Oc,P,p} =(\Gc_1 . \mathcal{CM}(\Oc,P,p)) =  \sum_{x\in\Gc_1\cap \mathcal{CM}(\Oc,P,p)} \textrm{length}\left( \Oc_{\Gc_1\cap \mathcal{CM}(\Oc,P,p),x}\right)=\sum_{\substack{(\bar{E},\iota) \in \textrm{SS}_{\Oc}(\rho) \\ \pm \ell_0 \in \bar{\LG} \subset \Ends(\bar{E}^2,\lambda_0 \iota(P)) \\ q/4(\ell_0)=1}} e_{\ell_0}.$$
 When running over $\textrm{Emb}_{\Oc}(p)$ instead of $\textrm{SS}_{\Oc}(\rho)$, Lemma~\ref{lem:emb} shows that we also have a factor $1,1/2$ or $2$ which plays a role. Hence the total contribution can be expressed as follows.

\begin{theorem} \label{th:mult}
    One has $$e_{\Oc,P,p}= \sum_{\substack{\iota \in \textrm{Emb}_{\Oc}(p)\\ \pm \ell_0 \in \bar{\LG}, \, q/4(\ell_0)=1}} \eps(\iota) \cdot e_{\ell_0}$$ where for $\iota : \Oc \hookrightarrow \Rc$ and $\Pc \subset \Rc$ the two-sided ideal over $p$, we define $\eps(\iota)$ as
   $$ \begin{cases}
        1/2 & \textrm{if } p \; \textrm{is inert and } \Pc \; \textrm{is principal}; \\
        2 & \textrm{if } p \; \textrm{is ramified and } \Pc \; \textrm{is not principal};\\
        1 & \textrm{otherwise}.
    \end{cases}$$
\end{theorem}

\subsection{An algorithm for the primes of PDR and their arithmetic intersection multiplicity}

We summarise the previous considerations in Algorithm~\ref{algo:PDR} and give some examples.

\begin{algorithm}
    \caption{Primes of potentially decomposable reduction and their multiplicity.}
\label{algo:PDR}
\begin{algorithmic}[1]
    \Require A maximal order $\Oc=\Z[\omega]$ of discriminant $-d<0$ in an imaginary quadratic field and a matrix $P \in M_2(\Oc)$ representing an indecomposable principal polarization on $E^2$ where $E$ is an elliptic curve with CM by $\Oc$.
    \Ensure The set $p$ of primes of potentially decomposable reduction and their intersection multiplicity $e_{\Oc,P,p}$.
    \State Use Theorem~\ref{th:finitep} or sieving with the conditions~\eqref{eq:sieve} to produce a finite list $\mathbb{P}$ of all possible primes of potentially decomposable reduction. Include $p=d/4$ if $d \equiv 0 \pmod{4}$ and $d/4$ is prime.
    \ForAll{$p  \in \mathbb{P}$}
    \ForAll{$(\iota: \Oc \hookrightarrow \Rc) \in \textrm{Emb}_{\Oc}(p)$}   
    \State Compute $\iota(\omega) \in \Rc$. We still denote it $\omega$.
      \State Write a basis of $\Rc$ as $1,\omega,r_1,r_2$.
      \State Compute the image of $P$ by $\iota$  in $M_2(\Rc)$. We still denote it $P$.
      \State Complete the basis $\ell_1,\ell_2,\ell_3$ of $\LG$ from Algorithm~\ref{algo:primereduction} into a basis of $\bar{\LG}$ by adding the images by $\tau$ of 
      $$P^{-1} \begin{pmatrix} 0 & r_1 \\ \bar{r_i} & 0 \end{pmatrix} \; \textrm{and }  
    P^{-1} \begin{pmatrix} 0 & r_2 \\ \bar{r_2} & 0 \end{pmatrix}.$$ 
      \If{ there exists $\ell_0$ (unique up to a sign) in $\bar{\LG}$ with $q/4(\ell_0)=1$}
        \State Compute the Gram matrix $Q_{\ell_0}$ of $\ell_0^{\perp} \subset \langle \ell_1,\ell_2,\ell_3,\ell_0 \rangle$.
        \State Compute the Gross-Keating invariants $(a_1,a_2,a_3)$ of $Q_{\ell_0}$.
        \State Compute the multiplicity $e_{\ell_0}$ using Formula~\eqref{eq:GK}.
        \If{$p$ is inert in $\Oc$}
        \If{the two-side ideal $\Pc$ over $p$ in $\Rc$ is principal}
        \State $e_{\ell_0}:=e_{\ell_0}/2$;
        \EndIf
        \Else
\If{the two-side ideal $\Pc$ over $p$ in $\Rc$ is not  principal}
\State $e_{\ell_0}:=2 e_{\ell_0}$;
        \EndIf
    \EndIf
    \EndIf
    \EndFor
    \State Compute $e_{\Oc,P,p}=\sum e_{\ell_0}$. 
    \EndFor
    \State \Return the list $(p,e_{\Oc,P,p})$ for which $e_{\Oc,P,p}>0$.
\end{algorithmic}
\end{algorithm}

\begin{example} \label{ex:fin163}
    Let us compute the intersection multiplicity for our running example with $d=163$. The possible primes of PDR were $\mathbb{P}=\{2,3,5,7,17,19,23\}$ and we checked in Example~\ref{ex:prime17} that $p=17$ is a prime of PDR as we found $\ell_0$ such that $q/4(\ell_0)=1$. The Gram matrix of $\langle \ell_1,\ell_2,\ell_3,\ell_0\rangle$ for $q/4$ is $$\begin{pmatrix}
         165 & -241 &  317 &   -3 \\
        -241 &  353 & -463 &   5 \\
         317 &-463 & 613&   -5 \\
          -3 &   5 &  -5  &  1
    \end{pmatrix}, \; \textrm{and } Q_{\ell_0}=\begin{pmatrix}
         156 &-226 & 302 \\
     -226 & 328 &-438 \\
       302& -438&  588
    \end{pmatrix}.$$ 
The matrix $Q_{\ell_0}$ has Gross-Keating invariants $(0,0,1)$ which corresponds to $e_{\ell_0}=1$. The second embedding of $\Oc$ in $\Rc$ is the conjugated one which gives the same value and by Theorem~\ref{th:mult}, we get $e_{\Oc,P,p}=(1+1)/2=1$ since $p$ is inert and $\Pc$ is principal.
\end{example}

\begin{example} \label{ex:79}
Let us take $d=79$. The ring $\Oc$ has class number 5 and there are 3 indecomposable principal polarizations $P_1=[3,7,\omega],P_2=[5,16,2\omega-1]$ and $P_3=[8,12,2\omega+3]$ on $E^2$. Let us pick the inert prime $p=7$. There is a unique maximal order $\Rc$, the ideal $\Pc$ is principal and there are 10 embeddings $\iota: \Oc \hookrightarrow \Rc$ up to $\Rc^{\times}$. But there are respectively $6,0$ or $4$ possible $\pm \ell_0$ for $P_1$, resp. $P_2$, resp. $P_3$, each with Gross-Keating invariants $(0,0,1)$. Hence for $i=1,2$ and $3$, $e_{\Oc,P_i,p}=3,0$ and $2$.
\end{example}

\begin{example}
Let $d=55$. There are two indecomposable polarizations $P_1=[3,5,\omega]$ and $P_2=[5,12,2\omega+1]$. Only for $P_1$, the prime $p=5$ is a prime of PDR. There is only one $\Rc \subset \Bc_{p,\infty}$ with 2 possible embeddings up to $\Rc^{\times}$, each leading to an element $\pm \ell_0$ of norm $1$. Their $Q_{\ell_0}$ matrices are equal, namely
$$Q_{\ell_0} = \begin{pmatrix}
  4 &-20 &  2 \\
-20 &112 &-10 \\
  2 &-10 &  4
\end{pmatrix}$$
with Gross-Keating invariants $(0,1,1)$ contributing each with a multiplicity $2$. Hence we get $e_{\Oc,P_1,p}=2+2=4$ since $\Pc$ is principal.
\end{example}

\section{Factorization of the discriminant of genus $2$ curves}
\label{sec:disc}

In \cite[Thm. 1]{igusa79}, Igusa computes generators for the ring  of Siegel modular forms for $\Sp_4(\Z)$ with integral Fourier coefficients. More precisely, let 
$\Ac_{2}$ be the moduli stack of principally polarised abelian schemes of relative dimension $2$. Let $\pi : \Vc_{2} \longrightarrow \Ac_{2}$ be the universal abelian scheme and $\pi_{*} \Omega^{1}_{\Vc_{2}/\Ac_{2}} \longrightarrow \Ac_{2}$ the rank $g$ bundle, usually called \emph{Hodge bundle}, induced by the relative regular differential forms of degree one on $\Vc_{2}$ over $\Ac_{2}$. The relative canonical bundle over $\Ac_{2}$ is the line bundle
$$
\wm =\bigwedge^{2} \pi_* \Omega^{1}_{\Vc_{2} / \Ac_{2}}.
$$
Igusa computes 15 generators for 
$$
S_{2} = \bigoplus_{h \in \N} \Gamma(\Ac_{2} \otimes \Z,\wm^{\otimes h}),
$$
the space of (geometric) Siegel modular forms over $\Z$.
The generators are expressed as polynomials in the algebraic theta constants defined by Mumford \cite[Appendix I]{mumford-tata3}. As explained in \cite[Sec.3]{ionicamodular}, by pulling back these sections to the space of Teichm\"uller forms for the moduli stack of curves $\mathcal{M}_2$, one can give expressions of the modular forms in terms of Igusa invariants $\underline{J}=(J_2,J_4,J_6,J_8,J_{10})$. After the correct choice of basis of differentials, we can make this map an equality, as Igusa does in his formula \cite[p. 848]{igusag3} and as we will do as well to keep the notation simple. We start with four modular forms in $S_2$ 
$\underline{\varphi}=(\varphi_4, \varphi_6, \varphi_{10}, \varphi_{12})$: $\varphi_m$ for $m=4,6$ are the Einsenstein modular forms $\psi_m$ in  \cite[p.189]{igusa-thetag2};  $\varphi_{10}=-2^2\chi_{10}$ and $\varphi_{12}=2^2\cdot3\chi_{12}$ where $\chi_m$ are the normalised cusp forms in p. 195 in \textit{loc.~cit.}. They all are polynomials in the theta constants with coefficients in $\Z[1/2]$ (some theta constants can be divided by $1/2$ and still have integral Fourier coefficients, see \cite[p.415]{ichi3}). Instead of the locus of Jacobians, one can also choose to restrict to the products of elliptic curves, as described in \cite[p. 168]{igusa79}. 
All this allows Igusa to give a $\operatorname{Spec}(\mathbb{Z})$-scheme isomorphisms between the coarse moduli space of curves $M_2$ and an open sub-scheme of an affine scheme in $A_2$ \cite[Thm. 2]{igusag2}, and similarly for products of elliptic curves in \cite[Lemma 15]{igusa79}. We unify these two loci in this section. For simplicity in the presentation, we first present the case away from characteristic $2$ or $3$.

\begin{proposition}\label{invA2}  Let $\mathbb{Z}[\frac{1}{2},\frac{1}{3}]
[\underline{\varphi}]$
with graduation given by the index of the variables, and let $X=\operatorname{Proj} \mathbb{Z}[\frac{1}{2},\frac{1}{3}][\underline{\varphi}]$.
There is an isomorphism between the coarse moduli space of principally polarised abelian surfaces $A_2$ and the  open set $D(\varphi_{10})\cup D(\varphi_{12})$ of $X$. 
\end{proposition}

\begin{proof}
The moduli space $A_2$ decomposes into the union of the locus of Jacobian of genus 2 curves and the products of elliptic curves with the product polarization, which corresponds to the locus $\Gc_1$ introduced previously. Recall that $\varphi_{10}=-2^2\cdot\chi_{10} =- \frac{1}{2^{12}} \prod_{\epsilon \; \textrm{even}} \theta[\epsilon]^2$ where $\theta[\epsilon]$ are the theta constants. This is a primitive Siegel modular form over $\Z$ and its zero locus inside $A_2$ is $2 \cdot \Gc_1$.

Hence, if $\varphi_{10}\neq0$, we work over the locus of Jacobians and following \cite[p. 848]{igusag3}, we have the relations:
\begin{align}
\varphi_4=J_2^2-2^3\cdot3\cdot J_4,\, \varphi_6=J_2^3-2^2\cdot3^2\cdot J_2J_4+2^3\cdot3^3\cdot J_6,\, \nonumber\\ \varphi_{10}= J_{10}\text{ and } \varphi_{12}=J_2J_{10}. \label{psiJ}
\end{align}
They realise an isomorphism between  the principal open  $D(\varphi_{10})$ of $X$ and the  coarse moduli space of smooth genus 2 curves ${M}_2$.

 If $\varphi_{10}=0$, we consider the geometric point $E\times E'$ where $E$ (resp. $E'$) has invariants $(c_4:c_6:\Delta)$ (resp. $(c'_4:c'_6:\Delta')$) as in \cite[p. 42]{silverman}. The product $E \times E'$ is uniquely determined by $(c_4c'_4:c_6c'_6:\Delta\Delta')$ since $j\cdot j'=\frac{(c_4c'_4)^3}{\Delta\Delta'}$ and $j+j'=1728+\frac{(c_4c'_4)^3-(c_6c'_6)^2}{1728\Delta\Delta'}$.
 The 10 even theta constants at $E \times E'$ are products of the theta constants for each elliptic curve, see for instance \cite[Chap.1, Thm. 11]{rauch}. Hence, an easy computation using the expressions of the Siegel modular forms in terms of theta constants in  \cite[p. 848]{igusag3} gives \begin{equation}\label{psiE}\varphi_4=c_4c'_4,\, \varphi_6=c_6c'_6,\, \varphi_{10}=0\text{ and }\varphi_{12}=2^2\cdot3\cdot\Delta\Delta'\neq0.
 \end{equation}
These formula realise an explicit isomorphism between $\Gc_1$ and the locus of polarised product of elliptic curves. The last 2 sets of equations allow to extend the isomorphism on the full $A_2$.
\end{proof}

For the case of characteristics $2$ and $3$, we also need to include the following Siegel modular forms with, despite a first impression, have integral Fourier coefficients: 
\begin{align*}
\varphi_{24} & =2^{-3}\cdot3^{-1}\cdot(\varphi_{12}^2-\varphi_{10}^2\varphi_4),\\
\varphi_{36} & = 2^{-3}\cdot3^{-3}\cdot(\varphi_6\varphi_{10}^3-\varphi_{12}^3+2^2\cdot3^2\cdot\varphi_{12}\varphi_{24}),   \\
\varphi_{48} & =2^{-2}\cdot(\varphi_{12}\varphi_{36}-\varphi_{24}^2).
\end{align*}

\begin{proposition}\label{prop:23}
Let $\mathbb{Z}[\underline{\varphi}']=\mathbb{Z}[\varphi_4,\varphi_6,\varphi_{10},\varphi_{12},\varphi_{24},\varphi_{36},\varphi_{48}]/(R_{24}, R_{36}, R_{48})$ with graduation given by the index of the variables and the relations being given by:
\begin{align}
R_{24}=2^3\cdot3\cdot\varphi_{24}-\varphi_{12}^2+\varphi_{10}^2\varphi_4, \nonumber\\
R_{36}=2^3\cdot3^3\cdot\varphi_{36}-\varphi_6\varphi_{10}^3+\varphi_{12}^3-2^2\cdot3^2\cdot\varphi_{12}\varphi_{24},\,R_{48}=2^2\cdot\varphi_{48}-\varphi_{12}\varphi_{36}-\varphi_{24}^2.
\end{align}
 Set  $X'=\operatorname{Proj} \mathbb{Z}[\underline{\varphi'}]/(R_{24},R_{36},R_{48})$
There is an isomorphism between the coarse moduli space of principally polarised abelian surfaces $A_2$ and the  open set $D(\varphi_{10})\cup D(\varphi_{12})\cup D(\varphi_{24})\cup D(\varphi_{36})\cup D(\varphi_{48})$, of $X'$.
\end{proposition}

\begin{proof}
The proof follows the same lines as the one of Proposition~\ref{invA2}. We only need the new identities:
\begin{itemize}
    \item For Jacobians of genus 2 curves 
\begin{align}
\varphi_4=J_2^2-2^3\cdot3\cdot J_4,\, \varphi_6=J_2^3-2^2\cdot3^2\cdot J_2J_4+2^3\cdot3^3\cdot J_6,\, \nonumber\\ 
\varphi_{10}=J_{10},\,\varphi_{12}=J_2J_{10},\,\varphi_{24}=J_4J_{10}^2,\,\varphi_{36}=J_6J_{10}^3,\,\varphi_{48}=J_8J_{10}^4.
\end{align}
\item For products of elliptic curves
\begin{align}
\varphi_4=c_4c'_4,\, \varphi_6=c_6c'_6,\, \varphi_{10}=0,\,\varphi_{12}=2^2\cdot3\cdot\Delta\Delta',\nonumber \\
\varphi_{24}=2\cdot3\cdot\Delta^2{\Delta'}^2,\,\varphi_{36}=2^2\cdot\Delta^3{\Delta'}^3,\,\varphi_{48}=3\Delta^4{\Delta'}^4.
\end{align}
\end{itemize}
\end{proof}

\begin{remark} In \cite[Thm. 1]{igusa79}, Igusa needs more Siegel modular forms to generate the ring $S_2$ than we use in Proposition \ref{prop:23}. He does indeed describe the Satake compactification of $A_2$, \cite[Sec. 9]{igusa-thetag2},  when
we describe here a strict sublocus since we do not consider points for which $\phi_{10}=\phi_{12}=0$.
\end{remark}

The previous isomorphisms allow us to established a simple relation between the normalised valuations of the Siegel modular forms and of the invariants over a discrete valuation ring $\Z_{\frakp}$ with valuation $v$. We recall that we defined 
the normalised valuation of $J_{10}$ at $\frakp$ by $$v_{\Jv}(J_{10})= v_{\frakp}(J_{10})-10 \cdot \min\left(\frac{v_{\frakp}(J_{n})}{n}, \, n \in \{2,4,6,8,10\}\right).$$
Similarly, for the set $\underline{\varphi'}$ (or $\underline{\varphi}$), we can define
$$v_{\underline{\varphi'}}(\varphi_{10})= v_{\frakp}(\varphi_{10})-10 \cdot \min\left(\frac{v_{\frakp}(\varphi_{n})}{n}, \, n \in \{4,6,10,12,24,36,48\}\right).$$

\begin{proposition}\label{prop:siegel} Let $\Z_{\frakp}$ be a discrete valuation ring with valuation $v$.  Let $x \in A_2/\Z_{\frakp}$ be a point corresponding to the Jacobian of a stable model of a  genus $2$ curve defined over $\Z_{\frakp}$ with decomposable reduction.  Then, 
$$
v_{\Jv}(J_{10}(C))=6v_{\underline{\varphi'}}(\varphi_{10}(x)).
$$
\end{proposition}
\begin{proof} We do the proof when the residue field is of characteristic different from 2 or 3, the proof is simply more cumbersome for the latter cases as one needs to consider the extra Siegel modular forms in Proposition~\ref{prop:23}. By \eqref{psiE}, since the special fiber is the product of two elliptic curves, the geometric situation is described by 
$$
v(\varphi_4(x))\geq4v,\,v(\varphi_6(x))\geq6v,\,v(\varphi_{10}(x))=10v+w,\,v(\varphi_{12}(x))\geq12v,
$$
with at least one of the inequalities being an equality. So, we get $v_{\underline{\varphi}'}(\varphi_{10}(x))=w$.
Reversing the equalities \eqref{psiJ} we obtain:
\begin{align}
J_2=\varphi_{12}/\varphi_{10},\,J_4=2^{-3}\cdot3^{-1}(\varphi_{12}^2/\varphi_{10}^2-\varphi_4), \nonumber \\
J_6=2^{-3}\cdot3^{-3}(\varphi_6+2^{-1}\cdot3^2\varphi_4\varphi_{12}/\varphi_{10}-2^{-1}\varphi_{12}^3/\varphi_{10}^3),\,J_{10}=\varphi_{10}. \label{Jpsi}
\end{align}
Hence,
$$
v(J_2(C))\geq2v-w,\,v(J_4(C))\geq4v-2w,\,v(J_6(C))\geq 6v-3w\text{ and }v(J_{10}(C))=10v+w,
$$
with at least one of the inequalities being an equality.
So, we conclude that $$v_{\Jv}(J_{10}(C))=10v+w-5(2v-w)=6w$$ and the result follows.

\end{proof}

It remains to connect these normalised valuations with the local intersection multiplicity in Section~\ref{sec:exp}. For sake of simplicity, we explain here again the case of characteristic of the residue field different from 2 and 3. Proposition~\ref{invA2} shows that locally around $x \in \Gc_1$, we can take  $\varphi_{12}=1$ and the points we are interested in have local ring $$\Oc_{\varphi_{10}\cap \mathcal{CM}(\Oc,P,p),x} = \Z_{\frak{p}}[\varphi_4,\varphi_6,\varphi_{10}]/(\varphi_4-x_1,\varphi_6-x_2,\varphi_{10}-x_3, \varphi_{10}) = \Z_{\frakp}/\frakp^{v(x_3)}\Z_{\frakp}.$$
 Now $v(x_3)$ is the normalised valuation $v_{\underline{\varphi}}(\varphi_{10}(x))$ and since $2 \cdot \Gc_1=(\varphi_{10})$, we see that $v_{\underline{\varphi}}(\varphi_{10}(x))$ is twice the length of the local ring $\Oc_{\Gc_1 \cap \mathcal{CM}(\Oc,P,p),x}$. Hence
 \begin{equation}\label{eq:totalmul}
 v_{\Jv}(J_{10}(C)) = 6 \cdot v_{\underline{\varphi}}(\phi_{10}(x)) = 12 \cdot \textrm{length}(\Oc_{\Gc_1 \cap \mathcal{CM}(\Oc,P,p),x})    
 \end{equation}
We therefore conclude with our final result.

\begin{theorem} \label{th:disc}
Let $\Oc$ be a maximal order in a quadratic imaginary field $K$ of discriminant $-d<0$. Let $E$ be an elliptic curve with CM by $\Oc$ defined over the Hilbert class field $H$ of $K$. Let $(E^2,\lambda_0 P)$ be a principally polarised abelian surface corresponding to the Jacobian of a curve $C$ over $H$. Let $p$ be a prime of potential decomposable reduction. If $p$ is inert in $\Oc$ or $p=d/4$ (resp. $p$ is ramified and different from $d/4$), then the exponent at $p$, $\sum_{\frakp | p} v_{\Jv}(J_{10}(C))$, of the absolute normalised discriminant is $12  \cdot e_{\Oc,P,p}$ (resp. $6  \cdot e_{\Oc,P,p}$). 
\end{theorem}

\begin{proof} Using \eqref{eq:totalmul}, we would simply get a factor 12 if running on the Galois conjugates of $C$. But in this section, we decided to fix the curve $C$ and to run through the $\frakp | p$ instead. We therefore need Lemma~\ref{lem:decomp} to see the correspondence between the $\frakp |p$ ad the Galois conjugates. This translates into:
$$
\sum_{\frak{p}\mid p} v_{\Jv}(J_{10}(C)) =\begin{cases}\sum\limits_{\sigma \in \Gal(H/K)}v_{\Jv}(J_{10}(C^{\sigma})), 
 \text{ if }p\text{ is inert in }\mathcal{O},\\
 \sum\limits_{\sigma \in \Gal(H/K) } v_{\Jv}(J_{10}(C^{\sigma}))/2,
 \text{ if }p\text{ is ramified and }p \neq d/4,\\
 \sum\limits_{\sigma \in \Gal(H/K)} v_{\Jv}(J_{10}(C^{\sigma})), 
 \text{ if }p= d/4.\end{cases}
$$
\end{proof}

\begin{example} \label{ex:facdis}
    Let us conclude with our running example for $d=163$. In Example~\ref{ex:fin163}, we got $e_{\Oc,P,17}=1$. Since $17$ is inert, we shall get $12 \cdot 1$ for  the exponent of the absolute normalised discriminant at 17. Now $D=J_{10}(C)$ and since $\frakp=(17)$ is inert in $H=K$, we see that $v(J_2)=v(J_4)=v(J_6)=v(J_8)=0$ so that $v(D)=v_{\Jv}(J_{10}(C))=12$ and the results are coherent. 
    
    For the interesting case of $p=2$, we get 
    \begin{multline*}
        (v(J_2(C)),v(J_4(C)),v(J_6(C)),v(J_8(C)),v(J_{10}(C)))=( 6, 9, 14, 16, 32) \\ = (2+2 \cdot 2, 1+2 \cdot 4, 2 + 2 \cdot 6, 0 + 2 \cdot 8, 12 + 2 \cdot 10)
    \end{multline*} and therefore $v_{\Jv}(J_{10}(C))=12$ without introducing the correction factor $2^{-12}$ in Equation~\eqref{eqD}.
\end{example}

\section{Experimental results, observations and conjectures} \label{sec:datum}
In this section, we first run experiments  for $\Oc$ of class number $1$ (Table~\ref{tab:comp})
and then for $d<100$ (Tables~\ref{tab:comp2} and \ref{tab:comp3}). For each $\Oc$, we pick an elliptic curve $E$ with CM by $\Oc$. For each indecomposable principal polarization $\lambda$ on $E^2$, we compute heuristically  a curve $C$ which Jacobian is isomorphic to $(E^2,\lambda)$, see Section~\ref{sec:explicit}.  We then compute the positive exponents of the absolute normalised discriminant, $\sum_{\frakp | p} v_{\Jv}(J_{10})$ at all primes $p$.  In parallel, we run our Algorithm~\ref{algo:PDR} to compute the primes of PDR for $(E^2,\lambda)$ and their multiplicity of intersection. We then compare these two computations thanks to Theorem~\ref{th:disc} and see that they match in all cases. The tables also offer a comparison with the third author's formula. The implementation in \texttt{Magma} of the various necessary algorithms can be found in the ancillary files attached to the \texttt{Arxiv} version of this article.

\subsection{Explicit construction.} \label{sec:explicit} Giving a maximal order $\Oc=Z[\omega]$ of an imaginary quadratic field $K$ of discriminant $-d$, an elliptic curve $E$ with CM by $\Oc$, we compute all possible indecomposable principal polarizations $\lambda$ on $E^2$ thanks to an update of the package created in \cite{gelin} (see Section~\ref{sec:pol}). For each polarization, represented by a matrix $P \in M_2(\Oc)$, we do the following
\begin{enumerate}
    \item following \cite{Rit-serre}, we compute an approximation of a period matrix for $(E^2,\lambda_0 P)$.
    \item We use the package \texttt{CHIMP} and the reconstruction algorithms  of \cite{sijsling-endo},  to compute a genus $2$ curve $C/H$ whose Jacobian is isomorphic to $(E^2,\lambda_0 P)$ over the Hilbert class field $H$ of $K$. As we use approximation of complex numbers, the result is heuristic. 
    \item We compute the Igusa invariants $\Jv=(J_2,J_4,J_6,J_8,J_{10})$ of $C$.
    \item for each prime $p$ dividing the norms relative to $H/\Q$ of the numerator of $J_{10}$ or of any of the denominators of $J_2,J_4,J_6$  or $J_8$, we compute its decomposition in prime ideals $\frakp$ in $H$ and at each $\frakp$, the multiplicity $e_{\frakp}=v_{\Jv}(J_{10})$ in $H_{\frakp}$.
    \item One returns the pair $(p,\sum_{\frakp} e_{\frakp})$ for which the second term is positive.
\end{enumerate}
We managed to run this algorithm for all $d<100$ with a precision of $300$ bits for the complex field. We could also handle all $d<200$ for which the class number of $K$ was less or equal to $10$. We used the package \texttt{MagmaPolred} from \href{https://github.com/edgarcosta/MagmaPolred}{Edgar Costa's github} which calls PARI-GP to reduce the coefficients of the class polynomials. 

\begin{remark}
In the primitive genus $2$ CM case (where the abelian surface is simple), bounds for the primes of bad reduction and their exponents in the discriminant were obtained  in \cite{goreng2,gola}. They were used in  \cite{streng,strengrecip} to certified the computation of Igusa Class polynomials. Similarly, our bounds may be used to certify the previous computations.
\end{remark}

\subsection{A close formula} In the 90', the third author worked out a close formula to obtain similar results, at least when $d$ is odd. In order to prove and maybe improve this formula in a future work, we show how it compares to the present results. The matrix $Q$ appearing in \eqref{eq:villegas163} is \emph{not} equivalent to the Gram matrix we obtain in Example~\ref{ex:gram163}. It can be cooked up in the following way. Starting with a Riemann matrix $[1,\omega]$ of $E/\C$, one can use \cite{Rit-serre} to compute algebraically a Riemann matrix $\tau=\begin{pmatrix}
        \tau_1 & \tau_2 \\ \tau_2 & \tau_3 
    \end{pmatrix}$ of $(E^2,\lambda)$. One can then derive the \emph{singular relations} from Humbert, see for instance \cite[Chap. 2]{gruenewald}. It gives a rank 3 lattice $L_1$ of $\Z^5=A\Z+B\Z+C\Z+D\Z+E\Z$ defined by the integral solutions of 
    $$A \tau_1 + B \tau_2 + C \tau_3 + D (\tau_2^2- \tau_3 \tau_1) +E=0.$$
    The lattice $\Z^5$ has the natural quadratic form $\Dc(A,B,C,D,E) = B^2-4AC-4DE$ which agrees with Humbert definition. However, this one does not give the correct $Q$ and one has to use the form $B^2-AC-DE$ instead, as in \cite{kudla}. We call $Q_1$ the Gram matrix of $L_1$ with respect to this modified quadratic form.
  We let $Q = -d \cdot Q_1^{-1}$ and we will compute 
\begin{equation} \label{eq:villegageneral}
-6 \sum_{m \in \Z^3} \sum_{n | (d-Q(m))/4} \left(\frac{n}{d}\right) \log n.    
\end{equation}

\begin{remark}
    It seems that there may be another, maybe more natural, interpretation of the matrix $Q_1$ as the Gram matrix of the lattice $L_0=\{\alpha \in \Ends(E^2,\lambda), \, \Tr(\alpha)=0\}$ with respect to our quadratic form $q/4$.
\end{remark}

  \begin{table}[]
{\small 
\begin{tabular}{|c|c|c|c|}
\hline
         $d$           & $P$ & Explicit computation & Formula~\eqref{eq:villegageneral}  \\ \hline
$4 \cdot 2$ &    $ [2,2,\omega+1]$          &   $2^{12}$  &   $-$         \\  \hline
$11$ &   $ [2,2,\omega] $          & $2^{12} $    &  $2^{12}$           \\ \hline
$19$ &   $[2,3,\omega]$            &   $2^{12} 3^{12}$  &     $2^{12} 3^{12}$           \\ \hline
\multirow{2}{*}{43} &   $  [2,6,\omega] $         &  $2^{24} 3^{12} 5^{12}$   &  $2^{24} 3^{12} 5^{12}$           \\  
                    &  $  [3,4,\omega]$           & $ 2^{12} 3^{12} 5^{12} 7^{12} $  &  $ 2^{12} 3^{12} 5^{12} 7^{12} $           \\ \hline
\multirow{3}{*}{67} &     $[2,9,\omega]$          &  $2^{12} 3^{12} 5^{12} 11^{12}$   & $2^{12} 3^{12} 5^{12} 11^{12}$      \\ 
                    &    $ [3,6,\omega]   $       &  $2^{12} 3^{12} 5^{12} 7^{12} 13^{12}  $  &   $2^{12} 3^{12} 5^{12} 7^{12} 13^{12}$    \\ 
                    &      $[4,5,\omega+1]$         &   $2^{24} 3^{12} 5^{12} 7^{12} 11^{12}$ &   $2^{24} 3^{12} 5^{12} 7^{12} 11^{12}$ \\ \hline
\multirow{7}{*}{163} &    $[2,21,\omega]$           & $2^{24}3^{12}5^{12}23^{12}29^{12}$     & $2^{24}3^{12}5^{12}23^{12}29^{12}$       \\  
                    &    $[3,14,\omega]$           &  $2^{12}3^{24}5^{12}7^{12} 11^{12} 17^{12} 31^{24}$   &  $2^{12}3^{24}5^{12}7^{12} 11^{12} 17^{12} 31^{24}$          \\
                    &    $[4,11,\omega+1]$           &  $2^{12} 3^{12} 5^{12} 7^{12} 11^{12} 13^{12} 23^{12} 29^{12}$    & $2^{12} 3^{12} 5^{12} 7^{12} 11^{12} 13^{12} 23^{12} 29^{12}$    \\                  
                    &   $[5,33,2\omega]$            & $2^{24} 3^{12} 5^{12} 7^{12} 13^{12} 17^{12} 19^{12} 37^{12}$    & $2^{24} 3^{12} 5^{12} 7^{12} 13^{12} 17^{12} 19^{12} 37^{12}$    \\                  
                    &    $[6,7,\omega]$           & $2^{12} 3^{24} 5^{12} 7^{12} 11^{12} 17^{12} 19^{12} 23^{12}$    &  $2^{12} 3^{24} 5^{12} 7^{12} 11^{12} 17^{12} 19^{12} 23^{12}$   \\                  
                    &    $[6,8,\omega+2]$           &  $2^{24}3^{12}5^{12}7^{12}11^{12}13^{12}29^{12}31^{12}$   &  $2^{24}3^{12}5^{12}7^{12}11^{12}13^{12}29^{12}31^{12}$   \\ 
                    &    $[7,24,2\omega+1]$           & $2^{36} 3^{12}5^{12}11^{12}13^{12}17^{12}19^{12}23^{12}$    &  $2^{36} 3^{12}5^{12}11^{12}13^{12}17^{12}19^{12}23^{12}$    \\ \hline
\end{tabular}
}
\caption{Comparison of explicit computations with Formula~\eqref{eq:villegageneral}. The indecomposable principal polarizations $P=\begin{pmatrix}
    a & b \\ \bar{b} & c
\end{pmatrix}$ are given by the lists $[a,c,b]$. Class number $1$ cases.}
\label{tab:comp}
\end{table}

\begin{table}[]
\begin{tabular}{|c|c|c|c|}
\hline
       $d$             & $P$ & Explicit computation & Formula~\eqref{eq:villegageneral}  \\ \hline
$4\cdot 5$ &       $[2,3,\omega]$        &  $2^{24}$   &     $-$      \\ \hline
23 &     $[3,3,\omega+1]$          &  $5^{12}$   &  $5^{12}$              \\\hline
$4\cdot 6$ &       $[2,4,\omega+1]$        &  $2^{12}3^{12}$   &   { 1 $*$}     \\ \hline
31 &       $[3,3,\omega]$        &  $3^{36}$    & $3^{36}$             \\ \hline
\multirow{2}{*}{35} &    $[2,5,\omega]$           & $2^{24}5^{12}$    &   { $2^{24}$ $*$}          \\
                    &   $[3,4,\omega+1]$            &  $2^{24}5^{12}$    &   { $2^{24}$ $*$}       \\ \hline
39 &       $[5,8,2\omega+1]$        &  $3^{36}$    & { $1$ $*$}           \\ \hline
\multirow{2}{*}{$4 \cdot 10$} & $[ 2, 6, \omega + 1 ]$ & $2^{12} 3^{24}$ & $-$  \\
                                & $[ 3, 4, \omega + 1 ]$ & $2^{24} 3^{24} 5^{12}$ & $-$   \\ \hline   
\multirow{2}{*}{$47$} &  $[3,5,\omega+1]$            &   $5^{36}$  &   $5^{36}$        \\ 
                  &    $[7,17,3\omega+2]$           &   $5^{12}11^{12}$  &  $5^{12}11^{12}$            \\ 
                   \hline
\multirow{2}{*}{$51$} &  $[2,7,\omega+1]$            &   $2^{36}3^{24}$  &   { $2^{36}$ $*$}        \\ 
                  &    $[4,4,\omega+1]$           &   $2^{36}3^{24}$  & { $2^{36}$ $*$}            \\ 
                   \hline
\multirow{2}{*}{$4 \cdot 13$} &  $[2,7,\omega+]$            &   $2^{24}3^{24}$  &    $-$        \\ 
                  &    $[3,5,\omega+1]$           &   $2^{24}3^{24}5^{24}$  &    $-$       \\ 
                   \hline
\multirow{2}{*}{$55$} &  $[3,5,\omega]$            &   $3^{48}5^{24}$  &  { $2^{-12}3^{48}$ $**$}        \\ 
                  &    $[5,12,2\omega+1]$           &   $3^{48}11^{12}$ &  { $3^{48}$ $*$}            \\ 
                   \hline
\multirow{3}{*}{$4\cdot 14$} &  $[2,8,\omega+1]$             &   $2^{24} 7^{12}$  &   $-$      \\ 
                    &    $[3,5,\omega]$           &   $2^{48}7^{12}$  &      $-$      \\ 
                    &          $[4,4,\omega+1]$     &     $2^{24}7^{24}$ &      $-$   
                    \\
                 \hline
\multirow{4}{*}{$59$} &  $[2,8,\omega]$            &   $2^{48}11^{12}$  &   $2^{48}11^{12}$        \\ 
                  &    $[3,6,\omega+1]$           &   $2^{36}13^{12}$ &  $2^{36} 13^{12} $        \\ 

                  &    $[4,4,\omega]$           &   $2^{36}11^{12}$ &  $2^{36} 11^{12}$          \\ 

                  &    $[5,12,2\omega-1]$           &   $2^{48}$ &  $2^{48}$            \\ 
                   \hline
\multirow{3}{*}{$4\cdot 17$} &  $[2,9,\omega]$             &   $2^{84}$  &   $-$     \\ 
                    &    $[3,6,\omega]$           &   $2^{60}5^{24}$  &  $-$   \\ 
                    &          $[5,14,2\omega+1]$     &     $2^{60}5^{24}$ &    $-$    \\ \hline
\multirow{3}{*}{$71$} &  $[3,7,\omega+1]$             &  $7^{24} 13^{12}$   &   $7^{24} 13^{12}$       \\ 
                    &    $[5,5,\omega+2]$           &   $11^{12}17^{12}$  &  $11^{12} 17^{12}$        \\ 
                    &          $[7,23,3\omega-1]$     &   $7^{36}11^{12}$  &     $7^{36} 11^{12}$      \\ \hline
\multirow{3}{*}{$79$} &  $[3,7,\omega]$             & $3^{60} 7^{36}$    &    $3^{60} 7^{36}$       \\ 
                    &    $[5,16,2\omega-1]$           &  $3^{60}17^{12}$   &   $3^{50} 17^{12}$        \\ 
                    &          $[8,12,2\omega+3]$     &  $3^{60} 7^{24}$   &     $3^{60} 7^{24}$      \\ \hline
\multirow{5}{*}{$83$} &   $[2,11,\omega]$           &   $2^{48}5^{36}$  &   $2^{48}5^{36}$         \\ 
                    &    $[3,8,\omega+1]$           &  $2^{36}5^{12}13^{12}$  &  $2^{36}5^{12} 13^{12}$             \\ 
                    &    $[4,6,\omega+1]$           &   $2^{36}5^{36} 13^{12}$ &  $2^{36}5^{36} 13^{12}$        \\ 
                    
                    &    $[5,17,2\omega]$           &   $2^{48}5^{12}19^{12}$  &  $2^{48}5^{12}19^{12}$            \\

                    &    $[7,12,2\omega-1]$           &   $2^{36}5^{12}$  &  $2^{36}5^{12}$           \\\hline
\end{tabular}
\caption{Comparison of explicit computations with Formula~\eqref{eq:villegageneral}.  The indecomposable principal polarizations $P=\begin{pmatrix}
    a & b \\ \bar{b} & c
\end{pmatrix}$ are given by the lists $[a,c,b]$. Class number $>1$ cases with $d<100$.}
\label{tab:comp2}
\end{table}

\begin{table}[]
\begin{tabular}{|c|c|c|c|}
\hline
       $d$             & $P$ pol. & Explicit computation & Formula~\eqref{eq:villegageneral}  \\ \hline
\multirow{2}{*}{$4 \cdot 21$} &  $[2,11,w]$            &   $2^{48}3^{24}$  &    $-$       \\ 
                  &    $[5,17,2w]$           &   $2^{48}3^{24}$ &     $-$      \\ 
                   \hline
\multirow{2}{*}{$87$} &  $[5,5,w+1]$   & $3^{48} 5^{48}$  &  { $2^{-12}5^{48}$ $**$}       \\ 
                  &    $[7,29,3w+1]$ &   $3^{48} 5^{48}$   &  { $2^{-12}5^{48}$ $**$}  \\ 
                   \hline
\multirow{4}{*}{$4 \cdot 22$} &  $[2,12,w+1]$            &   $2^{12}3^{24}11^{12}$  &      $-$      \\ 
                  &    $[3,8,w+1]$           &   $2^{36}3^{24}5^{24}7^{12}17^{12}$ &    $-$       \\ 

                  &    $[4,6,w+1]$           &   $2^{12}3^{24}5^{24}7^{24}11^{12}$ &     $-$        \\ 

                  &    $[5,18,2w+1]$           &   $2^{36}3^{24}5^{24}7^{12}17^{12}$ &    $-$          \\ 
                   \hline
\multirow{4}{*}{$91$} &  $[2,12,w]$            &  $2^{24}3^{24}17^{12}$  &    { $2^{12}3^{24} 17^{12}$ $**$}       \\ 
                  &    $[3,8,w]$           &   $2^{24}3^{24}13^{12}17^{12}$ & { $2^{24} 3^{24}  17^{12}$ $*$}          \\ 

                  &    $[4,6,w]$           &   $2^{24}3^{24}13^{12}17^{12}$ &  { $2^{24} 3^{24}  17^{12}$ $*$}     \\ 

                  &    $[5,6,w+2]$           &   $2^{24}3^{24}17^{12}$ &  { $2^{12} 3^{24}  17^{12}$ $**$}           \\ 
                   \hline
\multirow{4}{*}{$95$} &  $[3,9,w+1]$             &  $5^{24} 7^{48}$ &   {$2^{-12} 7^{48}$ $**$}        \\ 
                   &    $[5,5,w]$           & $7^{48} 19^{12}$    &   { $7^{48}$ $*$}        \\ 
                    &          $[5,20,2w+1]$     &  $5^{48} 7^{48}$   &    { $2^{-12} 3^{-12} 7^{48}$ $**$}       \\
                    &          $[7,31,3w]$     &   $5^{24} 7^{48}$  &     { $2^{-12} 7^{48}$ $**$}       \\
                    \hline
\end{tabular}
\caption{Comparison of explicit computations with Formula~\eqref{eq:villegageneral}. The indecomposable principal polarizations $P=\begin{pmatrix}
    a & b \\ \bar{b} & c
\end{pmatrix}$ are given by the lists $[a,c,b]$. Class number $>1$ cases with $d<100$ (sequel).}
\label{tab:comp3}
\end{table}

\subsection{Observations} \label{sec:observations}
Formula~\eqref{eq:villegageneral} works amazingly well when $d$ is a prime. When $d$ is an odd composite number, it fails detecting primes that divide $d$ (cases denoted $*$ in the table) which may lead to some side effects that we observe on other primes (cases denoted $**$ in the table).

We also observe in our experiments that the primes output by the refined sieve with the congruences relations \eqref{eq:sieve} and which are inert or ramified in $\Oc$ seems to actually be the set of primes of PDR directly, with the possible exception of having an extra prime dividing $d$. In other words,  we conjecture that if a matrix 
\begin{equation} 
T =\begin{pmatrix} \Gram(\LG) & \begin{array}{c}x \\ y \\ z\end{array}\\\begin{array}{ccc}x & y & z\end{array} & 1\end{pmatrix} >0
\end{equation}
with $(x,y,z) \in \Z^3$ satisfying the congruences \eqref{eq:sieve} has  a determinant divisible by a prime $p \nmid d$ then $p$ is a prime of PDR. For $p | d$, it is less clear as 
we found exceptions for $d <200$. For certain polarizations the pairs  $$(d,p)=\{(40,5),(91,7),(104,13),(120,5),(136,17),(168,7),(184,7),(187,11),(195,13)\}$$
do appear for matrices $T$ but $p$ is not a PDR for the respective polarization.
If the conjecture is true, it means that we can detect a prime of PDR simply in terms of these conditions without doing computations with the embeddings of $\Oc$ in the quaternion orders. However, as observed in Remark~\ref{rem:trueT}, some matrices $T$ do not come from true solutions in $\Ends(\bar{E}^2,\bar{\lambda})$. A possible refinement would be to keep only the matrices $T$ for which the Gross-Keating invariants do exist: this would exclude the second matrix of the Remark~\ref{rem:trueT} since its Gross-Keating invariants are $(0,0,2)$ and this would give a fractional value for $e_{\ell_0}$. However, this condition does not reduce the list of exceptions above.

\bibliographystyle{alpha}

\bibliography{synthbib}

\end{document}